\documentclass[11pt]{amsart}

\usepackage{amsmath, amssymb, amscd, latexsym, graphicx, psfrag, stmaryrd, MnSymbol}

\usepackage{color}

\newcommand{\si}{\sigma}
\newcommand{\Si}{\Sigma}
\newcommand{\Ga}{\Gamma}

\newcommand{\bK}{\mathbb{K}}
\newcommand{\bC}{\mathbb{C}}
\newcommand{\bD}{\mathbb{D}}

\newcommand{\bH}{\mathbb{H}}
\newcommand{\bL}{\mathbb{L}}
\newcommand{\bP}{\mathbb{P}}
\newcommand{\bQ}{\mathbb{Q}}

\newcommand{\bR}{\mathbb{R}}
\newcommand{\bS}{\mathbb{S}}
\newcommand{\bT}{\mathbb{T}}
\newcommand{\tbT}{{\widetilde{\mathbb{T}}}}
\newcommand{\bZ}{\mathbb{Z}}
\newcommand{\balpha}{{\boldsymbol \alpha}}
\newcommand{\bbeta}{{\boldsymbol \beta}}
\newcommand{\bSi}{{\boldsymbol \Si}}

\newcommand{\cA}{\mathcal{A}}
\newcommand{\cB}{\mathcal{B}}
\newcommand{\cC}{\mathcal{C}}
\newcommand{\cD}{\mathcal{D}}
\newcommand{\cE}{\mathcal{E}}
\newcommand{\cG}{\mathcal{G}}

\newcommand{\cI}{\mathcal{I}}

\newcommand{\cL}{\mathcal{L}}
\newcommand{\cM}{\mathcal{M}}
\newcommand{\cO}{\mathcal{O}}

\newcommand{\cS}{\mathcal{S}}
\newcommand{\cT}{\mathcal{T}}

\newcommand{\cX}{\mathcal{X}}
\newcommand{\cY}{\mathcal{Y}}
\newcommand{\IX}{\mathcal{IX}}

\newcommand{\fB}{\mathfrak{B}}
\newcommand{\fr}{\mathfrak{r}}
\newcommand{\fs}{\mathfrak{s}}
\newcommand{\fm}{\mathfrak{m}}
\newcommand{\fp}{\mathfrak{p}}
\newcommand{\fl}{\mathfrak{l}}
\newcommand{\fn}{\mathfrak{n}}
\newcommand{\fg}{\mathfrak{g}}
\newcommand{\fC}{\mathfrak{C}}

\newcommand{\const}{\mathrm{const}}
\newcommand{\Spec}{\mathrm{Spec}}

\newcommand{\Hom}{\mathrm{Hom}}
\newcommand{\age}{{\mathrm{age}}}
\newcommand{\Hess}{{\mathrm{Hess}}}
\newcommand{\orb}{\mathrm{orb}}
\newcommand{{\inv} }{\mathrm{inv}}
\newcommand{\ev}{\mathrm{ev}}
\newcommand{\Aut}{\mathrm{Aut}}
\newcommand{\Res}{\mathrm{Res}}

\newcommand{\val}{ {\mathrm{val}} }
\newcommand{\vir}{{\mathrm{vir}}}
\newcommand{\CR}{  {\mathrm{CR}}  }
\newcommand{\Area}{\mathrm{Area}}

\newcommand{\Jac}{\mathrm{Jac}}
\newcommand{\End}{\mathrm{End}}

\newcommand{\one}{\mathbf{1}}

\newcommand{\be}{\mathbf{e}}
\newcommand{\bu}{\mathbf{u}}
\newcommand{\bt}{\mathbf{t}}
\newcommand{\bGa}{\mathbf{\Ga}}
\newcommand{\btau}{\boldsymbol\tau }

\newcommand{\bmu}{\boldsymbol\mu}

\newcommand{\Int}{\mathrm{Int}}

\newcommand{\brho}{{\boldsymbol{\rho}}}

\newcommand{\w}{\mathsf{w}}
\newcommand{\su}{\mathsf{u}}
\newcommand{\sv}{\mathsf{v}}
\newcommand{\sw}{\mathsf{w}}
\newcommand{\sC}{\mathsf{C}}

\newcommand{\tC}{ {\widetilde{C}}}
\newcommand{\tF}{ {\widetilde{F}}}

\newcommand{\tQ}{ {\widetilde{Q}} }
\newcommand{\tS}{ {\widetilde{S}} }

\newcommand{\tL}{{\widetilde{L}}}
\newcommand{\tcL}{{\widetilde{\cL}}}

\newcommand{\tw}{{\widetilde{w}}}

\newcommand{\tX}{{\widetilde{X}}}
\newcommand{\txi}{ {\widetilde{\xi}} }

\newcommand{\tgamma}{{\widetilde{\gamma}}}
\newcommand{\tmu}{\widetilde{\mu}}

\newcommand{\hX}{\hat{X}}
\newcommand{\hY}{\hat{Y}}
\newcommand{\hx}{\hat{x}}
\newcommand{\hy}{\hat{y}}

\newcommand{\vGa}{\vec{\Ga}}
\newcommand{\vmu}{\vec{\mu}}
\newcommand{\vs}{{\vec{s}}}

\newcommand{\oC}{{\overline C}}

\newcommand{\Vbar}{{\overline V}}

\newcommand{\BSi}{ \mathrm{Box}(\Si)}

\newcommand{\eff}{\mathrm{eff}}
\newcommand{\Ker}{\mathrm{Ker}}

\newcommand{\Hol}{{\mathcal{H}ol}}
\newcommand{\bsi}{{\boldsymbol{\si}}}
\newcommand{\bgamma}{{\boldsymbol{\gamma}}}
\newcommand{\spa}{ {\ \ \,} }

\newcommand{\Mbar}{\overline{\cM} }

\newcommand{\ST}{ {S_{\bT}} }
\newcommand{\RT}{ {R_{\bT}} }
\newcommand{\bST}{ {\bar{S}_{\bT}} }
\newcommand{\bSTQ}{ \bST[\![ \tQ,\tau'' ]\!] }

\newcommand{\nov}{\Lambda_{\mathrm{nov}}}
\newcommand{\novT}{\bar{\Lambda}^{\bT}_{\mathrm{nov}} }

\newtheorem{dummy}{dummy}[section]

\newtheorem{theorem}[dummy]{Theorem}

\newtheorem{corollary}[dummy]{Corollary}
\newtheorem{proposition}[dummy]{Proposition}

\newtheorem{definition}[dummy]{Definition}
\newtheorem{remark}[dummy]{Remark}
\newtheorem{example}[dummy]{Example}

\allowdisplaybreaks[1]

\begin{document}

\title{Graph sums in the Remodeling Conjecture}

\author{Bohan Fang}
\address{Bohan Fang, Beijing International Center for Mathematical Research, Peking University, 5 Yiheyuan Road, Beijing 100871, China}
\email{bohanfang@gmail.com}


\author{Zhengyu Zong}
\address{Zhengyu Zong, Yau Mathematical Sciences Center, Tsinghua University, Jin Chun Yuan West Building,
Tsinghua University, Haidian District, Beijing 100084, China}
\email{zyzong@math.tsinghua.edu.cn}

\begin{abstract}
The BKMP Remodeling Conjecture \cite{Ma,BKMP09,BKMP10} preditcs all genus open-closed Gromov-Witten invariants for a toric Calabi-Yau $3$-orbifold by Eynard-Orantin's topological recursion \cite{EO07} on its mirror curve.
The proof of the Remodeling Conjecture by the authors \cite{FLZ1,FLZ3} relies on comparing two Feynman-type graph sums in both A and B-models. In this paper, we will survey these graph sum formulae and discuss their roles in the proof of the conjecture.
\end{abstract}

\maketitle

 \tableofcontents
\section{Introduction}

\subsection{Outline of the proof}
The Remodeling Conjecture can be viewed as an all genus open-closed mirror symmetry for toric Calabi-Yau 3-orbifolds. On A-model side, one has the higher genus open-closed Gromov-Witten potential for a toric Calabi-Yau 3-orbifold. On B-model side, the higher genus B-model potential comes from the Eynard-Orantin topological recursion on the mirror curve. The Remodeling Conjecture identifies the A-model and B-model higher genus potentials under the mirror map. At first glance, the Remodeling Conjecture may seem to be mysterious. However, we will see that the proof of the Remodeling Conjecture is quite natural under the point of view of quantization of semisimple Frobenius structures. In this subsection, we give the outline of the proof of the  Remodeling Conjecture.

Recall that the genus zero mirror theorem \cite{CCIT} works for general semi-projective toric orbifolds. On A-model side, one considers an $n-$dimensional semi-projective toric orbifold $\cX$ over $\bC$. There is an $n$-dimensional torus $\bT\cong(\bC^*)^n$ in $\cX$ as a dense open subset. The natural $\bT$-action on itself extends to a $\bT$-action on $\cX$. Under this action, there are finitely many fixed points and finitely many 1-dimensional $\bT-$invariant orbits. In practice, one may choose a smaller torus $\bT'\subset \bT$ acting on $\cX$ such that fixed points and the 1-dimensional $\bT'-$invariant orbits are the same as those of the $\bT$-action.

One can consider the $\bT'-$equivariant orbifold quantum cohomology $QH^*_{\CR,\bT'}(\cX)$ of $\cX$. The quantum product $\star_{\btau}$ on $QH^*_{\CR,\bT'}(\cX)$ is defined by the genus 0 Gromov-Witten invariants of $\cX$. There is also a $\bT'-$ equivariant Poincare pairing $(\cdot,\cdot)_{\cX,\bT'}$ on $\cX$. Then one obtains a Frobenius algebra
$$
(QH^*_{\CR,\bT'}(\cX),\star_{\btau},(\cdot,\cdot)_{\cX,\bT'}).
$$

On B-model side, one has the Landau-Ginzburg mirror of $\cX$. More concretely, there is a $\bT'$-equivariant superpotential
$$
W^{\bT'}: (\bC^*)^n\to \bC,
$$
which is determined by the toric data of $\cX$ and the $\bT'$-action on $\cX$. The $\bT'$-equivariant super potential $W^{\bT'}$ defines the Jacobi ring $\Jac(W^{\bT'})$:
$$
\Jac(W^{\bT'})=\bC[X_1^\pm,\cdots,X_n^\pm]/\langle X_1\frac{\partial W^{\bT'}}{\partial X_1},\cdots,X_n\frac{\partial W^{\bT'}}{\partial X_n}\rangle.
$$
There is a natural product $\cdot$ on $\Jac(W^{\bT'})$ induced from the product on $\bC[X_1^\pm,\cdots,X_n^\pm]$. One can also consider the residue pairing on the Jacobi ring $\Jac(W^{\bT'})$:
$$
(f,g):=\frac{1}{(2\pi\sqrt{-1})^n}\int_{|dW^{\bT'}|=\epsilon}\frac{fg\frac{dX_1\cdots dX_n}{X_1\cdots X_n}}{X_1\frac{\partial W^{\bT'}}{\partial X_1}\cdots X_n\frac{\partial W^{\bT'}}{\partial X_n}},
$$
where $f,g\in \Jac(W^{\bT'})$. These data form a Frobenius algebra on B-model:
$$
(\Jac(W^{\bT'}),\cdot,(\cdot,\cdot)).
$$
The genus 0 mirror theorem can be viewed as an identification of Frobenius structures on A-model and B-model: under the mirror map, we have an isomorphism of Frobenius algebras\footnote{We require $\btau\in H^{\leq 2}(\cX)$, i.e. we only consider \emph{small} quantum cohomology. The mirror symmetry is then established after the mirror map tranformation.}
$$
(QH^*_{\CR,\bT'}(\cX),\star_{\btau},(\cdot,\cdot)_{\cX,\bT'})\cong (\Jac(W^{\bT'}),\cdot,(\cdot,\cdot)).
$$
Both A-model and B-model Frobenius algebras are semisimple.

When we restrict the above mirror theorem to the case of toric Calabi-Yau 3-orbifolds, we obtain an identification of the Frobenius algebra on A-model and the Frobenius algebra on its Landau-Ginzburg mirror. On the other hand, the Remodeling Conjecture concerns the higher genus (open-closed) Gromov-Witten invariants instead of the quantum cohomology which is about the genus 0 data. So a natural question is: what is the higher genus B-model? In \cite{EO07}, Eynard and Orantin introduce the topological recursion on spectral curves. The output of the topological recursion is a symmetric $n$-form $\omega_{g,n}$ on the curve, where $g,n\geq 0$. In \cite{BKMP09} and \cite{BKMP10}, Bouchard, Klemm, Mari\~{n}o, and Pasquetti conjecture that when we apply the topological recursion to the mirror curve, we obtain the higher genus B-model potential $\check{F}_{g,n}$ is defined from $\omega_{g,n}$. The Remodeling Conjecture identifies the higher genus open-closed Gromov-Witten potential $F_{g,n}^{\cX,(\cL,f)}$ ($\cL$ is an Aganagic-Vafa brane and $f\in \bZ$ is the framing) of a toric Calabi-Yau 3-orbifold $\cX$ and the higher genus B-model potential $\check{F}_{g,n}$ under the open-closed mirror map. Thus the Remodeling Conjecture is an all genus open-closed mirror symmetry for toric Calabi-Yau 3-orbifolds.

The mirror curve as the B-model in the Remodeling Conjecture is related to the Landau-Ginzburg model by a dimensional reduction \cite{HV00,HIV00}. More concretely, for any toric \emph{Calabi-Yau} $3$-fold $\cX$, the mirror curve $C$ is defined by the following equation:
$$
C=\{(X,Y)|H(X,Y)=0\}\subset (\bC^*)^2,
$$
while the $\bT'$-equivariant super potential $W^{\bT'}:(\bC^*)^3\to\bC$ is given by
$$
W^{\bT'}=H(X,Y)Z-\su_1\log X-\su_2\log Y.
$$
Here $(\bC^*)^2\cong\bT'\subset\bT\cong(\bC^*)^3$ is the Calabi-Yau sub-torus which acts trivially on the canonical bundle of $\cX$ and $\su_1,\su_2$ are the equivariant parameters i.e. $H^*(\cB\bT')=\bC[\su_1,\su_2]$.

The key idea in the proof of the Remodeling Conjecture is that one can realize both A-model and B-model higher genus potentials as quantizations on the same semisimple Frobenius structure. On A-model side, this is given by a generalization of Givental formula to the orbifold case \cite{Zo}. In \cite{Zo}, the Givental formula is generalized to any GKM orbifolds (there are finitely many torus fixed points and finitely many 1-dimensional invariant orbits) and one can apply this formula to the case of toric Calabi-Yau 3-orbifolds which are a special kind of GKM orbifolds. Givental formula is expressed in terms of quantization of quadratic Hamiltonians which involves the exponential of quadratic differential operators.

On B-model side, the quantization procedure is  given by the topological recursion on the mirror curve. The higher genus data $\omega_{g,n}$ is obtained recursively from the initial data $\omega_{0,1},\omega_{0,2}$. By the dimensional reduction, the data $\omega_{0,1},\omega_{0,2}$ is equivalent to the data of the Frobenius structure $(\Jac(W^{\bT'}),\cdot,(\cdot,\cdot))$ of the Landau-Ginzburg B-model.

The bridge connecting the A-model quantization (orbifold Givental formula) and the B-model quantization (topological recursion) is the graph sum formula. On A-model side, one can apply Wick's formula to the Givental formula to rewrite this formula (involving differential operators) in terms of Feynman graphs. On B-model side, by \cite{E11} \cite{DOSS}, the topological recursion is equivalent to a graph sum formula. The Remodeling Conjecture is proved by identifying each factor in the graph sum formulas on A-model and B-model.

It turns out that all the factors in the graph sum formulas are determined by the $R$-matrix, which appears in the fundamental solutions of the quantum differential equation, and the disk potential $F_{0,1}^{\cX,(\cL,f)}$ ($\check{F}_{0,1}$ on B-model). The genus 0 mirror theorem identifies the Frobenius structures on A-model and B-model and hence identifies the quantum differential equations. By the uniqueness of the fundamental solution, the A-model and B-model R-matrices can be identified up to a constant matrix. On A-model side, this constant matrix is fixed by the orbifold quantum Riemann-Roch theorem \cite{Ts10} while on B-model side, this constant matrix is obtained by direct computation. It turns out that the A-model and B-model R-matrices are indeed equal. In the end, the identification of the disk potentials $F_{0,1}^{\cX,(\cL,f)}$ and $\check{F}_{0,1}$ is given in \cite{FLT}. Therefore, the Remodeling Conjecture follows immediately.

\subsection{Overview of the paper}
We fix the notation on toric orbifolds and Gromov-Witten invariants in Section \ref{sec:toric}.
In Section \ref{sec:b-model} we introduce $3$ related and equivalent B-models to a toric Calab-Yau $3$-orbifold $\cX$, with an emphasis on mirror curves. In Section \ref{sec:genus-0-mirror} we will give a quick review on the genus $0$ mirror theorem, and on the identification of the Frobenius structures. In Section \ref{sec:A-model-quantization}, we study the A-model quantization. The orbifold Givental formula expresses the higher genus Gromov-Witten potentials of a toric Calabi-Yau $3$-orbifold in terms of its Frobenius structure, which is the genus zero data. We also discuss the graph sum version of the orbifold Givental formula. In Section \ref{sec:B-model-quantization}, we move on to the B-model quantization, which is the topological recursion on a spectral curve. We also discuss the graph sum formula of the topological recursion. Then we specialize to the case of the mirror curve of a toric Calabi-Yau $3$-orbifold. In Section \ref{sec:comparing} we discuss the dimensional reduction from the Landau-Ginzburg B-model. Then we expresses the factors in the graph sum formula of B-model in terms of the B-model Frobenius structure. Then comparing the graph sum components leads to a complete proof of the Remodeling Conjecture. We will illustrate many facts by a running example.

\subsection{Acknowledgments}

We would like to thank Chiu-Chu Melissa Liu for the wonderful collaboration towards a proof of the BKMP Remodeling Conjecture, and her constant encouragement during the writing of this survey. We would also like to thank Bertrand Eynard for pointing out the relation between the graph sums in \cite{E11} and \cite{DOSS}.

\subsection{Table of notations}
\begin{center}
\begin{tabular}{| c | p{6cm} | p{5cm}|}
\hline
Symbols    & Explanation & Remark\\
\hline
$\cX$ & a toric CY $3$-(orbi)fold, defined by a defining polytope $P$ and its triangulation & fixed throughout the paper; fan $\Si$ is a cone over $P$\\
\hline
$\bT$  & torus acting on the toric CY 3-fold $\cX$ & $\bT\cong (\bC^*)^3$ \\
\hline
$\bT'$ & Calabi-Yau torus $\subset \bT$ preserving the CY form & $\bT'\cong (\bC^*)^2$\\
\hline
$\cL$ & a fixed outer Aganagic-Vafa brane & location (phase) given by $\si_0\in \Si(3), \tau_0\in \Si(2)$\\
\hline
$I_\Si$ & set of canonical basis for $QH_{\bT}^*(\cX), H^*_{\CR,\bT}(\cX)$, etc. & $I_\Si=\{\bsi=(\si,\gamma)\},\si\in \Si(3),\gamma\in G^*_\si.$\\
\hline
$I_\cC$ & set of critical points on a spectral curve $\cC$ & $I_\cC=I_\Si$ by the mirror thm if $\cC=$ mirror curve\\
\hline
$P_\balpha$ & critical points of $W^{\bT'}$, in $(\bC^*)^3$ & \\
\hline
$p_\balpha$ & critical points of the mirror curve $C_q\subset (\bC^*)^2$ & $P_\balpha=(p_\balpha, Z_\balpha)$\\
\hline
$\Jac(W^{\bT'})$ & Jacobian ring of $W^{\bT'}$ & $\cong QH_{\bT'}^*(\cX)$ under the mirror map, as a Frobenius alg.\\
\hline
$V_\balpha$ & canonical basis of $\Jac(W^{\bT'})$&  $V_\balpha(P_\bbeta)=\delta_{\balpha\bbeta}$ \\
\hline
$\fp$ & $h^2_\CR(\cX)$, number of twisted K\"ahler parameters & also \#(number of integral points in $\Delta$) $-3$\\
\hline
$\fg$ & $h^4_\CR(\cX)$, also the genus of the compactified mirror curve $\oC_q$ & also number of integral points in $\mathrm{Int}(\Delta)$
\\
\hline
$q_a,\ a=1\dots \fp$ & complex parameters mirror to extended K\"ahler parameters & depending on a choice of extended K\"ahler basis $H_1,\dots, H_\fp$\\
\hline
$C_q$ & Affine mirror curve at parameter $q$ & \\
\hline
$\oC_q$ & Compactified mirror curve at parameter $q$ & \\
\hline
$F_{g,n}^{\cX,(\cL,f)}$ & A-model open GW potential & depends on $\cX$, $\cL$ and the framing $f$
\\
\hline
$\omega_{g,n}$ & B-model higher genus invariants from the  EO recursion& symmetric meromorphic form on $(\oC_q)^n$.
\\
\hline
$\check F_{g,n}$ & B-model open potential & defined as the indefinite integral of $\omega_{g,n}$  \\
\hline
\end {tabular}
\end{center}

\section{Geometry and the A-model of a toric Calabi-Yau $3$-orbifold}
\label{sec:toric}

The Remodeling Conjecture \cite{Ma,BKMP09,BKMP10} concerns the open-closed all genus Gromov-Witten invariants of a semi-projective toric Calabi-Yau $3$-orbifold. We fix the notations in this section.

\subsection{Toric Calabi-Yau 3-orbifolds}
A Calabi-Yau 3-fold $X$ is toric if it contains the algebraic torus
$\bT=(\bC^*)^3$ as a Zariski dense open subset, and the action of
$\bT$ on itself extends to $X$. All Calabi-Yau 3-folds are non-compact.
There is a rank 2 subtorus $\bT' \subset \bT$
which acts trivially on the canonical line bundle of $X$.
We call $\bT'$ the Calabi-Yau torus. Then $\bT\cong \bT'\times \bC^*$. Let $\bT'_\bR\cong U(1)^2$ be the
maximal compact subgroup of $\bT'$.

Let $M'=\Hom(\bT',\bC^*)\cong \bZ^2$  and $N'=\Hom(\bC^*,\bT')$ be the character lattice
and the cocharacter lattice of $\bT'$, respectively. Then $M'$ and $N'$
are dual lattices.  Let $X_\Sigma$ be a toric Calabi-Yau 3-fold defined by a
simplicial fan $\Si \subset  N'_\bR\times \bR$,
where $N'_\bR:=N'\otimes_{\bZ}\bR\cong \bR^2$ can be identified with
the Lie algebra of $\bT_\bR'$.
Then $X_\Si$ has at most quotient singularities. We assume that
$X_\Si$ is semi-projective, i.e., $X_\Si$ contains at least one $\bT$-fixed point, and $X_\Si$ is projective over its affinization
$X_0:=\Spec H^0(X_\Si,\cO_{X_\Si})$.  Then the support of the fan $\Si$ is
a strongly convex rational polyhedral cone $\Si_0\subset N'_\bR \times \bR \cong \bR^3$,
and $X_0$ is the affine toric variety defined by the 3-dimensional cone $\Si_0$.
There exists a convex polytope $P \subset N'_\bR\cong \bR^2$ with vertices in the lattice $N'\cong \bZ^2$, such
that $\Si_0$ is the cone over $P\times \{1\}\subset N'_\bR\times \bR $, i.e.
$\Si_0=\{ (tx, ty, t): (x,y)\in P, t\in [0,\infty)\}$.
The fan $\Si$ determines a triangulation of $P$:
the 1-dimensional, 2-dimensional, and 3-dimensional cones
in $\Si$ are in one-to-one correspondence with
the vertices, edges, and faces of the triangulation of $P$, respectively.
This triangulation of $P$ is known as the toric diagram or the dual graph of the simplicial
toric Calabi-Yau 3-fold $X_\Si$.

Let $\Si(d)$ be the set of $d$-dimensional cones in $\Si$, and let
$\fp'=|\Si(1)|-3$. We label $\rho_1,\dots,\rho_{\fp'+3}\in \Si(1)$ (and usually denote the their generators in $P\times \{ 1 \}$ by $b_1,\dots, b_{\fp'+3}$). Then $X_\Si$ is a GIT quotient
$$
X_\Sigma = \bC^{3+\fp'}\sslash G_\Si =  (\bC^{3+\fp'}- Z_\Si)/G_\Si
$$
where $G_\Si$ is a $\fp'$-dimensional subgroup of $(\bC^*)^{3+\fp'}$ and
$Z_\Si$ is a Zariski closed subset of $\bC^{3+\fp'}$ determined by
the fan $\Si$. If $X_\Si$ is a smooth toric Calabi-Yau 3-fold
then $G_\Si\cong (\bC^*)^{\fp'}$ and $G_\Si$ acts freely on
$\bC^{3+\fp'}-Z_\Si$. In general we have $(G_\Si)_0\cong (\bC^*)^{\fp'}$,
where $(G_\Si)_0$ is the connected component of the identity, and
the stabilizers of the $G_\Si$-action on $\bC^{3+{\fp'}}-Z_\Si$ are at most finite and
generically trivial. The stacky quotient
$$
\cX = [(\bC^{3+\fp'}- Z_\Si)/G_\Si]
$$
is a toric Calabi-Yau 3-orbifold; it is a toric Deligne-Mumford
stack in the sense of Borisov-Chen-Smith \cite{BCS05}. For any $\si\in \Si(d)$ there is a codimensional $d$ closed substack $\fl_\si$ associated to $\si$. We denote its generic stabilizer to be $G_\si$. When $\si \in \Si(3)$, then $\fl_\si$ is a $\bT$-fixed (probably stacky) point.

We denote the $\bar D_i$ to be the first Chern class of divisors $\widetilde D_i=\{Z_i=0\}$, where $Z_i$ are homogeneous coordinates for $i=1,\dots,\fp'$.

\subsection{Toric crepant resolution and extended K\"{a}hler classes}
Given a semi-projective simplicial toric Calabi-Yau 3-fold $X_\Si$ which is not
smooth, there exists a subdivision $\Si'$ of $\Si$, such that
$$
X_{\Si'} = (\bC^{3+\fp}-Z_{\Si'})/G_{\Sigma'} \longrightarrow
X_\Si = \big((\bC^{3+\fp'}-Z_\Si)\times (\bC^*)^\fs\big)/G_{\Si'}
$$
is a crepant toric resolution, where $X_{\Si'}$ is a smooth
toric Calabi-Yau 3-fold, $\fp+3=|\Si'(1)|$, and $G_{\Si'}\cong (\bC^*)^{\fp}$. We denote $\rho_{\fp'+4},\dots,\rho_{\fp+3}\in \Si'(1)\setminus \Si(1)$ as new $1$-cones in $\Si'$, and their generators in $P\times \{ 1 \}$ by $b_{\fp'+4},\dots, b_{\fp+3}$. $X_{\Si'}$ and $X_\Si$ are GIT quotients of the same $G_{\Si'}$-action
on $\bC^{3+\fp}$ with respect to different stability conditions. We set $b_i=(m_i,n_i,1)\in P\times\{1\}$.

Let $K_{\Si'}\cong U(1)^{\fp}$ be the maximal compact subgroup of $G_{\Si'}\cong (\bC^*)^{\fp}$.
The $G_{\Si'}$-action on $\bC^{3+\fp}$ restricts to a Hamiltonian $K_{\Si'}$-action on
the K\"{a}hler manifold $(\bC^{3+\fp}, \omega_0=\sqrt{-1} \sum_{i=1}^{3+\fp} dz_i \wedge d\bar{z}_i)$,
with moment map $\tmu: \bC^{3+\fp}\to \mathrm{Lie}(K_{\Si'})^\vee=H^2_{G_{\Si'}}(\bC^{3+\fp};\bR)=\bR^{\fp}$. There exist two (open) cones
$\sC$ and $\sC'$ in $\bR^{\fp}$ such that
\begin{align*}
&\quad\ \tilde{\mu}^{-1}(\vec{r})/K_{\Si'}\\[1ex]
&= \begin{cases}
(\bC^{3+\fp}-Z_{\Si'})/G_{\Sigma'} = X_{\Si'}, & \vec{r}\in \sC',\\
\big((\bC^{3+\fp'}-Z_{\Si})\times (\bC^*)^{\fp-\fp'}\big)/G_{\Sigma'} = (\bC^{3+\fp'}-Z_{\Si})/G_\Si = X_{\Si}, & \vec{r}\in \sC
\end{cases}
\end{align*}
Here $\sC'\subset \bR^{\fp}= H^2(X_{\Si'};\bR)$ is the K\"{a}hler cone of $X_{\Si'}$ and
$\sC\subset \bR^{\fp}$ is the extended K\"{a}hler cone of $X_\Si$. Let $D_i\in H^2_{G_{\Si'}}(\bC^{3+\fp};\bR)=\mathrm{Lie}(K_{\Si'})^\vee$ for $i=1,\dots,\fp$ be the equivariant Poincar\'e dual to $\{Z_i=0\}$ in $\bC^{3+\fp}$ ($Z_i$ are coordinates).

The parameter $\vec{r}\in \sC$ determines a K\"{a}hler form $\omega(\vec{r})$ on
the toric Calabi-Yau 3-orbifold $\cX=[(\bC^{3+\fp'}-Z_\Si)/G_\Si]$. As shown in \cite{Ir09}, there is a canonical decomposition
\begin{equation}
\mathrm{Lie}(G_{\Si'})^\vee\cong H^2(X_\Si;\bC)\oplus \bigoplus_{j=\fp'+4}^{\fp+3}\bC D_j.
\label{eqn:eK-splitting}
\end{equation}
In particular, we write $\bar H$ for the projection of $H\in \mathrm{Lie}(G_{\Si'})^\vee$ to $H^2(X_\Si;\bC)$. Our notation $\bar D_i\in H^2(X_\Si;\bC)$ indeed satisfies this convention, and in particular $\bar D_j=0$ for $j=\fp'+4,\dots, \fp+3$. This splitting also applies to the K\"aher cone
\[
\sC=(\text{K\"ahler cone of $X_{\Si}$})\\ \bigoplus(\sum_{i=\fp'+4}^{\fp+3} \bR_{>0} D_i).
\]

\begin{example}
  \label{exp:polytope}
The polytope $P$ and the triangulation is the following. The vertices are $(0,0),(0,2),(1,0),(2,-1)$ (Figure \ref{fig:defining-polytope}). The fan $\Si$ is a cone over the triangulated $P$. If one adds the dashed line to the triangulation, we get the fan $\Si'$.
\begin{figure}[h]
\begin{center}
\psfrag{si0}{\small $\si_0$}
\psfrag{si1}{\small $\si_1$}
\psfrag{si2}{\small $\si_2$}
\psfrag{tau0}{\small $\tau_0$}
\includegraphics[scale=0.6]{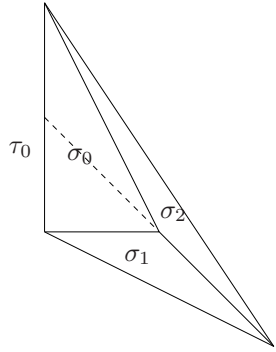}
\caption{The defining polytope of our main example $\cX$.}
\label{fig:defining-polytope}
\end{center}
\end{figure}
\end{example}

\subsection{Chen-Ruan orbifold cohomology}\label{sec:CR}
Let $U=\bC^{3+\fp'}-Z_\Si$, so that $\cX=[(\bC^{3+\fp'}- Z_\Si)/G_\Si]$. Given $v\in G_\Si$, let
$U^v = \{ z\in U: v\cdot z= z\}$.
The inertia stack of $\cX$ is
$$
\IX =\bigcup_{v\in \BSi} \cX_v
$$
where $\BSi=\{v\in G_\Si: U^v\neq \emptyset\}$ and $\cX_v =[U^v/G_\Si]$.

We consider cohomology with $\bC$-coefficient. As a graded $\bC$-vector space,
the Chen-Ruan orbifold cohomology \cite{CR04} of $\cX$ is
$$
H^*_\CR (\cX;\bC) = \bigoplus_{v\in \BSi} H^*(\cX_v;\bC)[2\age(v)], \quad \age(v)\in \{0,1,2\}.
$$
So
\[
H^2_\CR(\cX;\bC)=H^2(X_\Si;\bC)\bigoplus( \bigoplus_{\age(v)=1}\bC \one_{v}).
\]
One can show $H^2_\CR(\cX;\bC)\cong\mathrm{Lie}(G_{\Si'})^\vee\cong \bC^\fp$ as a vector space. Under this isomorphism $\one_{v}$ corresponds to a $D_i$  for $i=\fp'+4,\dots,\fp+3$ (c.f. Equation \eqref{eqn:eK-splitting}).

Let $\fg :=|\Int(P)\cap N'|$
be the number of lattice points in $\Int(P)$, the interior of the polytope $P$, and
let $\fn:=|\partial P\cap N'|$ be the number of lattice points on $\partial P$, the boundary of the polytope $P$.
Then
\begin{align*}
\fp' &= |\Si(1)|-3= \dim_\bC  H^2(X_\Si;\bC),\\
\fp&= |\Si'(1)|-3= |P \cap N'|-3=  \dim_{\bC}\mathrm{Lie}(G_{\Si'})^\vee=\dim_{\bC} H^2_\CR(\cX;\bC)\\
   &= \fg+\fn-3,\\
\fg &= |\Int(P)\cap N'| =\dim_\bC H^4_\CR(\cX;\bC),\\
N&:=\chi(\cX)= |\Si'(3)|= 2\Area(P)= \dim_\bC H^*_\CR(\cX;\bC)\\
 &= 1+ \fp+\fg = 2\fg-2+\fn.
\end{align*}

We choose $H_1,\dots, H_\fp$ in the closure of the extended K\"ahler cone $\bar \sC\subset \mathrm{Lie} (K_{\Si'})^\vee$ such that the following is true.
\begin{itemize}
\item $\{H_1,\ldots,H_\fp\}$ is a basis of $\mathrm{Lie} (K_{\Si'})^\vee=\bR^\fp$.
\item $\{\bar{H}_1,\ldots, \bar{H}_{\fp'}\}$ is a basis
of $H^2(X_{\Si};\bR)$. We require $H_i=\bar H_i$ under the identification \eqref{eqn:eK-splitting} for $i=1,\dots,\fp'$, and $H_i$ is in the K\"ahler cone of $X_{\Si}$.
\item $H_a= D_{3+a}$ for $a= \fp'+1,\dots,\fp$.
\item We choose $H_a$ in the lattice generated by all $D_i$.
\end{itemize}

The $\fp$ parameters
$\vec{r}=(r_1,\ldots,\allowbreak r_{\fp})$ are extended K\"{a}hler parameters
of $\cX$, where $r_1,\ldots, r_{\fp'}$ are K\"{a}hler parameters of $\cX$.
The A-model closed string flat coordinates are complexified
extended K\"{a}hler parameters
$$
\tau_a=-r_a+\sqrt{-1}\theta_a,\quad a=1,\ldots, \fp.
$$
\begin{example}[Example \ref{exp:polytope},continued]
  \label{exp:number-of-parameters}
  \begin{align*}
    &\cX=(\bC^4-Z_{\Si'})/G_{\Si'}=(\bC^5-Z_{\Si})/G_{\Si}, \\
    &\fp=2,\ \fp'=1,\ \fg=1,\ \fn=4,\ N=4,\\
    &\text{K\"ahler cone}=\bR_{>0} H_1,\\
    &\text{extended K\"ahler cone} \sC=\bR_{>0} H_1\oplus \bR_{>0}H_2.
  \end{align*}
  The moment map of $K_{\Si}$ is given by
  \[
  |Z_1|-4|Z_2|^2+|Z_3|^2+2|Z_4|^2,
  \]
  while the moment map $\tmu$ of $K_{\Si'}$ is given by
  \[
  (|Z_1|-4|Z_2|^2+|Z_3|^2+2|Z_4|^2,|Z_1|^2+|Z_3|^2-2|Z_5|^2).
  \]
  $H_1,H_2$ are $(1,0)$ and $(0,-1)$ respectively in the image of the moment map of $K_\Si$.
\end{example}

\subsection{Equivariant cohomology and its canonical basis}

\label{sec:equivariantCR}

We can work with the equivariant version of the Chen-Ruan (or ordinary) cohomology. We set the notions here for $\bT$-equivariant cohomology, while the notions for  $\bT'$-equivariant cohomology are self-evident if we change $\bT$ to $\bT'$.

The torus $\bT$ fits into the following exact sequence
\[
1\to G_\Si\to \tbT\cong (\bC^*)^{\fp+3} \to \bT\to 1,
\]
where $\tbT$ acts on $\bC^{\fp+3}$ in the standard way.

Let $R_\bT=H^*_{\bT}(\mathrm{pt})$. Then
\[
R_\bT=\bC[\su_1,\su_2,\su_3],\ S_\bT=\bC(\su_1,\su_2,\su_3).
\]
These $\su_1,\su_2,\su_3$ are basis of $M$, and characters of $\bT$. Choose $\su_3$ such that  $\bT'=\mathrm{ker}(\su_3)$.
Setting $\su_3=0$ passes into the equivariant setting for $\bT'$. We have $R_{\bT'}=H^*_{\bT'}(\mathrm{pt})$. Then
\[
R_{\bT'}=\bC[\su_1,\su_2],\ S_{\bT'}=\bC(\su_1,\su_2).
\]

One defines $D_i^\bT\in H_\tbT^2(\bC^{\fp+3};\bC)\cong H^2_{\CR,\bT}(\cX;\bC)$ as the $\tbT$-equivariant Poincar\'e dual of $\{Z_i=0\}\subset \bC^{\fp+3}$ -- it is a lift of $D_i$ into the equivariant cohomology. Similarly, we denote $\bar D_i^\bT$ as the equivariant first Chern class of $ \{ Z_i=0 \} \subset X_\Si$. This is an equivariant lift of $\bar D_i$. In particular, $\bar D_i^\bT=0$ for $i=\fp'+4,\dots,\fp+3$.

We have
\begin{equation*}
H^2_{\CR,\bT}(\cX;\bC)=H^2_\bT(X_\Si;\bC)\bigoplus( \bigoplus_{\age(v)=1} \bC \one_{v}).
\end{equation*}
We still denote the projection of $H$ to $H^2_\bT(X_\Si;\bC)$ by $\bar H$. We choose an equivariant lift $\bar H^\bT_a$ of $\bar H_a$ for $a=1,\dots,\fp'$. Then
\[
H^2_{\CR,\bT}(\cX;\bC)=\bigoplus_{i=1}^3 \bC\su_i \bigoplus_{a=1}^{\fp'}\bC \bar H_a^\bT \bigoplus( \bigoplus_{\age(v)=1} \bC \one_{v}).
\]
In the rest of this paper, whenever an equivariant (quantum, Chen-Ruan or classical) cohomology $QH^*_\bT, H^*_\bT, H^*_{\CR,\bT}$ is omitting the coefficient, we always regard it as over $\bC$.

We define the following ``extended Mori cone''.
\[
\bK_\eff=\bigcup_{\si\in \Si(3)}\{\beta \in \mathrm{Lie}(K_{\Si'})| \langle \beta, D_i\rangle\in \bZ_{\geq 0}, \Si(1)\ni\rho_i\notin \si\}.
\]
For any $\beta\in \bK_\eff$, $v(\beta)=\sum_{i=1}^{\fp+3} \lceil \langle D_i,\beta \rangle \rceil b_i$. We choose a finite extension field $\bST$ of $S_\bT$ such that $H^*_{\CR,\bT}(\cX;\bC)\otimes_{R_\bT} S_\bT$ is a semisimple algebra. We describe the canonical basis here.

For any $\si\in \Si(3)$, $\cX_\si=[\bC^3/G_\si]\subset \cX$ is an affine toric CY $3$-orbifold. An $\bST$-basis of $H^*_{\CR,\bT}(\cX_\si;\bC)\otimes_{R_\bT} S_\bT$ are $\one_h,\ h\in G_\si$, each corresponding to a connected component of $\cI \cX_\si$. Their products are
\[
 \one_h\cup_{\cB G_\si}  \one_{h'}= \one_{hh'}.
\]
For any $\gamma\in G_\si^*$, introduce
\[
\phi_\gamma=\sum_{h\in G_\si^*}\chi_{\gamma}(h^{-1})  \frac{\one_{h}}{\prod_{i=1}^3 \sw_i(\si)^{c^\si_i(h)}}.
\]
Here $\sw_i(\si)$ is the weight of the $\bT$ action on $T_{\fl_\si}\cX_\si$, while $h\in G_\si$ acts on the $i$-th factor of $(\bC^*)^3$ by multiplying $e^{2\pi\sqrt{-1}c^\si_i(h)}$. They are canonical basis of $H^*_{\CR,\bT}(\cX_\si;\bC)\otimes_{R_\bT} S_\bT$ \[\phi_\gamma\cup_{\cX_\si} \phi_{\gamma'}=\delta_{\gamma,\gamma'} \phi_{\gamma}.\] The sum of the pullbacks of the inclusion maps $\cX_\si\hookrightarrow \cX$ identifies
\[
H^*_{\CR,\bT}(\cX;\bC)\otimes_{R_\bT} S_\bT\cong \bigoplus_{\si\in \Si(3)} H^*_{\CR,\bT}(\cX_\si;\bC)\otimes_{R_\bT} S_\bT.
\]
The basis $\{ \phi_{\bsi}|\bsi=(\si,\gamma),\ \si\in \Si(3), \gamma\in G_\si^*\}$ is a canonical basis. We denote the set of A-model canonical basis by $I_\Si=\{\bsi=(\si,\gamma)|\si\in \Si(3), \gamma\in G_\si^*\}$, and $N=\# I_\Si=\dim_{\bC}H^*_{\CR}(\cX;\bC)$. Let $\phi^\bsi$ be the dual basis to $\phi_\bsi$ under the equvariant Poincar\'e pairing.

\subsection{Toric graphs}
The action of the Calabi-Yau torus $\bT'$ on $\cX$ restricts
to a Hamiltonian $\bT'_\bR$-action on the K\"{a}hler orbifold
$(\cX,\omega(\vec{r}))$, with moment map $\mu':\cX\to M'_\bR=\bR^2$.
The 1-skeleton $\cX^1$ of the toric Calabi-Yau 3-fold $\cX$ is the union of 0-dimensional
and 1-dimensional orbits of the $\bT$-action on $\cX$.
The image $\mu'(\cX^1) \subset \bR^2$ is a planar trivalent graph, which
is known as the toric graph of the symplectic toric Calabi-Yau 3-orbifold
$(\cX,\omega(\vec{r}))$. The toric graph depends
also on the symplectic structure of $\cX$.

\subsection{Aganagic-Vafa Lagrangian branes}
\label{sec:AV}
An Aganagic-Vafa Lagrangian brane in a toric Calabi-Yau 3-orbifold $\cX$ is a Lagrangian
sub-orbifold of the form
$$
\cL=[\tL/K_{\Si'}]\subset \cX=[\tmu^{-1}(\vec{r})/K_{\Si'}]
$$
where
\begin{align*}
 \tL = \Bigg\{ &(z_1,\ldots, z_{3+p+s})\in \tmu^{-1}(\vec{r}): \\
&\sum_{i=1}^{3+p+s}\hat{l}_i^1|z_i|^2=c_1, \sum_{i=1}^{3+p+s}\hat{l}_i^2|z_i|^2=c_2,\
 \arg(z_1\cdots z_{3+p+s}) = c_3 \Bigg\},
\end{align*}
$c_1,c_2,c_3$ are constants, and
$$
\sum_{i=1}^{3+p+s} \hat{l}^\alpha_i=0,\quad \alpha=1,2.
$$
The compact 2-torus $\bT'_\bR\cong U(1)^2$ acts
on $\cL$, and under its moment map $\mu'$, the image $\mu'(\cL)$ is a point on the toric graph $\Gamma=\mu'(\cX^1)$ and it
is not a vertex.  The Lagrangian $\cL$ intersects a unique $1$-dimensional $\bT$ orbit $\fl\subset \cX$ such that $\overline \fl=\fl_{\tau_0}$ where $\tau_0\in \Si(2)$.
We have $\fl\cong \bC^*\times \cB \bmu_\fm$ for some positive integer $\fm$ where $G_{\tau_0}\cong \bmu_\fm$. Here $\bmu_\fm\cong \bZ_\fm$ is a  multiplicative subgroup of $U(1)$. When $\fm=1$, $\cL\cong S^1\times \bC$ is smooth; when $\fm>1$,
$\cL$ is smooth away from $\cL\cap \fl \cong S^1\times \cB\bmu_m$. We require our Lagrangian $\cL$ is \emph{outer}, i.e. the closure of $\fl$ in $\cX$ is not compact ($\overline \fl=\bC\times \cB \bmu_\fm$). Then $\tau_0$ lies in a unique $\si_0\in \Si(3) $. By a rearrangement of order, we require $b_1,b_2,b_3$ span $\si_0$, while $b_2,b_3$ span $\tau_0$ and $b_1,b_2,b_3$ are labeled counterclockwisely as the vertices of $\si_0$. We have a short exact sequence of finite abelian groups
\[
1\to G_{\tau_0}\cong \bmu_\fm\to G_{\si_0}\to \bmu_{\fr}\to 1.
\]
The stabilizer at the $\bT$-fixed point $\fl_{\si_0}$ is $G_{\si_0}$ while the generic stabilizer on $\fl_{\tau_0}=\bar\fl$ is $G_{\tau_0}$.

\begin{figure}[h]
\begin{center}
\psfrag{L}{\tiny $\mathcal L$}
\psfrag{lt0}{\tiny $\fl_{\tau_0}$}
\psfrag{ls0}{\tiny $\fl_{\si_0}$}
\includegraphics[scale=0.3]{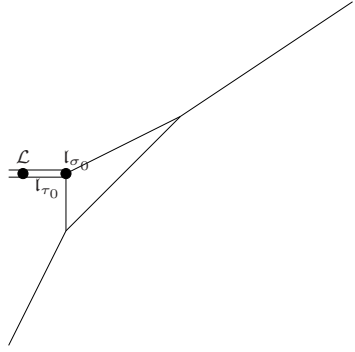}
\caption{The toric graph of our main example $\cX$. The gerby leg $\fl_{\tau_0}$'s image is a line, but we draw a double line to denote it is gerby.}
\label{fig:graph}
\end{center}
\end{figure}

\begin{example}[Example \ref{exp:number-of-parameters}, continued]
  \label{exp:graph}
  Given the choice of $\si_0$ and $\tau_0$ as in Figure \ref{fig:defining-polytope}, the toric graph for $\cX$ and the phase of the Aganagic-Vafa brane $\cL$ is in Figure \ref{fig:graph}. We have
  \begin{align*}
    &G_{\tau_0}\cong \bmu_2,\ \fm=2,\ G_{\si_0}\cong \bmu_2,\ \fr=1,\\
    & \fl_{\tau_0}\cong \bC\times \cB \bmu_2,\ \fl_{\si_0}=[\mathrm{pt}/\bmu_2],\\
    &G_{\si_1}=G_{\si_2}=\{1\},\ \fl_{\tau_1}=\fl_{\tau_2}=\mathrm{pt}.
  \end{align*}
\end{example}

There is also a \emph{framing} datum $f\in\bZ$. It prescribes a subtorus $\bT'_f\subset \bT'$ by $\bT'_f=\mathrm{ker}(\su_2-f \su_1).$ We denote $H^*_{\bT'_f}(\mathrm{pt})=\bC[\sv]$ such that $\sv$ is the image of $\su_1$ under the restriction, while the image of $\su_2$ is $f\sv$.

\subsection{Primary closed Gromov-Witten invariants and\\ A-model free energies}

We define genus $g$, degree $\beta$ primary closed Gromov-Witten invariants:
$$
\langle \gamma_1,\dots, \gamma_\ell \rangle^{\cX}_{g,\ell,\beta}=
\int_{[\Mbar_{g,\ell}(\cX,\beta)^{\bT'}]^\vir }
\frac{\prod_{i=1}^\ell \ev_i^*(\gamma_i) }{e_{\bT'}(N^\vir)}\\
\in R_{\bT'}.
$$
The A-model genus $g$ free energy $\tF_g^{\cX}$ is
a generating function of primary genus $g$ closed Gromov-Witten invariants, as a function of $\btau \in H^{2}_{\CR,\bT'}(\cX)$.
\begin{eqnarray*}
\tF_{g}^{\cX}(\btau)
&=& \sum_{\beta,\ell\geq 0}
\frac{\langle \btau,\dots,\btau \rangle^{\cX}_{g,\ell,\beta}}{\ell!}. \\
\end{eqnarray*}

\subsection{Open Gromov-Witten invariants and A-model open potentials}

The BKMP Remodeling Conjecture builds the mirror symmetry for the open Gromov-Witten potentials $\tF_{g,n}^{\cX,(\cL,f)}(\tX_1,\ldots, \tX_n,\btau)$ as well as free energies $\tF_{g}^{\cX}(\btau)$.

The Aganagic-Vafa Lagrangian brane $\cL$ is homotopic to $S^1\times \cB\bmu_\fm$, so
$$
H_1(\cL;\bZ)= \bZ\times \bmu_\fm.
$$

Open GW invariants of $(\cX,\cL)$  count holomorphic maps
$$
u: (\Sigma, x_1,\ldots, x_\ell, \partial \Sigma =\coprod_{j=1}^n R_j) \to (\cX,\cL)
$$
where $\Si$ is a bordered Riemann surface with stacky points $x_i=\cB\bZ_{r_i}$ and
$R_j\cong S^1$ are connected components of $\partial \Sigma$. These invariants depend on
the following data:
\begin{enumerate}
\item the topological type $(g,n)$ of the coarse moduli of the domain, where $g$
is the genus of $\Si$ and $n$ is the number of connected components of $\partial\Sigma$,
\item the degree $\beta'=u_*[\Si]\in H_2(\cX,\cL;\bZ)$,
\item the winding numbers $\mu_1,\ldots,\mu_n \in \bZ$ and the monodromies $k_1,\ldots, k_n \in \bmu_\fm$,
where $(\mu_j,k_j)= u_*[R_j]\in H_1(\cL;\bZ)=\bZ\times \bmu_\fm$,
\item the framing $f\in\bZ$  of $\cL$.
\end{enumerate}
We call the pair $(\cL,f)$ a framed Aganagic-Vafa Lagrangian brane.
We write $\vmu=((\mu_1,k_1),\ldots, (\mu_n,k_n))$. Here $k_i\in \{0,\dots, \fm-1\}$ is regarded as an element $e^{2\pi\sqrt{-1}k_i/\fm}\in \bmu_\fm$.  Let $\cM_{(g,n),\ell}(\cX,\cL\mid \beta',\vmu)$ be the moduli space\linebreak
parametrizing maps described above, and let $\Mbar_{(g,n),\ell}(\cX,\cL\mid\beta',\vmu)$
be the partial compactification: we allow the domain $\Si$ to have nodal singularities, and
an orbifold/stacky point on $\Si$ is either a marked point $x_j$ or a node; we require
the map $u$ to be stable in the sense that its automorphism group is finite. Evaluation
at the $i$-th marked point $x_i$ gives a map $\ev_i: \Mbar_{(g,n),\ell}(\cX,\cL\mid\beta',\vec{\mu})\to \IX$.

Given $\gamma_1,\ldots, \gamma_\ell\in H^*_{\CR,\bT'}(\cX;\bC)$, we define
\begin{align}
  \label{eqn:open-GW}
\langle \gamma_1,\ldots, \gamma_\ell\rangle_{g,\beta,\vec{\mu}}^{\cX,(\cL,f)}
&:= \int_{[\Mbar_{(g,n),\ell}(\cX,\cL\mid\beta',\vec{\mu})^{\bT_\bR'}]^\vir }
\left.\frac{\prod_{i=1}^\ell \ev_i^*\gamma_i }{e_{\bT'_\bR}(N^\vir)}\right|_{(\bT_f)_\bR}\\
&\in  \bC \sv^{\sum_{i=1}^\ell \frac{\deg\gamma_i}{2}-1}\nonumber
\end{align}
where $\bT_\bR'$ and $(\bT_f)_\bR$ are the corresponding real sub-torus of $\bT'$ and $\bT_f$, which preserves the Lagrangian $\cL$, $H^*_{\bT'_f}(\mathrm{pt})=\bC[\sv]$, $\beta\in H_2(\cX;\bZ)$ and  $\beta'=\beta+\sum{\mu_i}\in H_2(\cX,\cL;\bZ)$.

For $\btau\in H^2_{\CR,\bT'}(\cX;\bC)$, we define
generating functions $\tF_{g,n}^{\cX,(\cL,f)}$ of open Gromov-Witten invariants as follows.
\begin{align}
  \label{eqn:open-potential}
&\quad\ \tF_{g,n}^{\cX,(\cL,f)}(\tX_1,\ldots, \tX_n,\btau)\\
&= \sum_{\beta,\ell\geq 0}\sum_{(\mu_j,k_j)\in \bZ\times \bmu_\fm}
\frac{\langle \btau^\ell \rangle^{\cX,(\cL,f)}_{g,\beta,(\mu_1,k_1)\cdots,
 (\mu_n, k_n)}Q^\beta}{\ell!}
 \nonumber \\
&\quad\   \cdot \otimes_{j=1}^n \Big(\tX_j^{\mu_j} (-(-1)^\frac{-k_j}{\fm})\one'_\frac{-k_j}{\fm}\Big)
\in H^*_\CR(\cB\bmu_\fm;\bC)^{\otimes n}\nonumber
\end{align}
where $H^*_\CR(\cB\bmu_\fm;\bC)=\oplus_{k=0}^{m-1} \bC \one'_{\frac{k}{\fm}}$, and $\one'_{-\frac{k}{\fm}}=\one'_{\frac{\fm-k}{\fm}}$ for $k\in \{1,\dots, \fm-1\}$.
The closed Gromov-Witten invariants and potential can be viewed as a special case for $n=0$ i.e. there is no boundary on the domain curve. The variable $Q$ is the Novikov variable
\[
Q^\beta=Q_1^{\langle H_1,\beta\rangle}\dots Q_{\fp'}^{\langle H_{\fp'},\beta\rangle}.
\]

\subsection{Descendant closed Gromov-Witten invariants} Given
$\gamma_1,\ldots,\gamma_n$, we define a generating function of genus $g$, $n$-point
descendant closed Gromov-Witten invariants:
$$
\left\llangle \gamma_1\psi_1^{k_1},\ldots,\gamma_n\psi^{k_n}_n
\right\rrangle_{g,n}^{\cX}=\sum_{\beta,\ell\geq 0}\frac{Q^\beta}{\ell!}
\left\langle \gamma_1\psi^{k_1}_1,\ldots,\gamma_n\psi^{k_n},
\btau^\ell\right\rangle_{g,n+\ell,\beta}^{\cX},
$$
where $\psi_i=c_1(\bL_i)$ and $\bL_i\to \Mbar_{g,n+\ell}(\cX,\beta)$ is line bundle
whose fiber at moduli point $[u:(C,x_1,\ldots,x_{n+\ell})\to \cX]$ is
the cotangent line $T_{x_i}^*C$ at the $i$-th marked point to (the coarse moduli
space of) the domain curve. In this notation $\tF_g(\btau)=\llangle\rrangle_{g,0}^\cX$.


\section{Mirror curves and the Landau-Ginzburg mirror}
\label{sec:b-model}

\subsection{Three B-models}

The Remodeling Conjecture by Mari\~no and Bouchard-Klemm-Mari\~no-Pasquetti \cite{Ma,BKMP09,BKMP10} expresses all genus open-closed Gromov-Witten invariants in terms of the Eynard-Orantin recursion on its mirror curve. The mirror curve plays the central role in the B-model, and is one of three related and equivalent B-models.

\subsubsection{Landau-Ginzburg model}

Recall we have a choice of K\"ahler basis $\{H_1,
\dots, H_\fp\}$ as given in section \ref{sec:CR}. The non-equivariant superpotential is
\[
W=\sum_{i=1}^{\fp+3} X_i,
\]
where
\begin{equation}
  \label{eqn:LG-torus}
\prod_{i=1}^{\fp+3} X_i^{\langle D_i, \beta\rangle }= q^\beta,\ \forall \beta \in \mathbb K_\eff.
\end{equation}
The Equation \eqref{eqn:LG-torus} prescribe a $3$-dimenionsal algebraic torus $\cY\subset (\bC^*)^{\fp+3}$. We regard this $W$ as the superpotential in the LG-model on this algebraic torus $\cY$. The parameters $q_1,\dots,q_\fp$ are complex parameters of the B-model. We use $q_{K}=(q_1,\dots,q_{\fp'})$ to denote the parameters corresponding to the \emph{K\"ahler} part, while $q_\orb=(q_{\fp'+1},\dots,q_{\fp})$ denotes the \emph{twisted} part. Under mirror symmetry, heuristically $\log q_K$ measures the mutual distances between vertices (corresponding to $3$-cones, or torus fixed points in $\cX$) in the toric graph.

We define
\begin{align*}
& a_{m_i,n_i}=1,\ i=1,2,3,\\
&\prod_{a=1}^\fp(a_{m_{a+3},n_{a+3}}(q))^{\langle D_{a+3},\beta\rangle } =q^\beta,\forall \beta \in \bK.
\end{align*}
Under the large radius limit $q_K\to 0$, $a_{m_{a+3},n_{a+3}}(q)\to 0$ for $a=1,\dots, \fp'$. The coordinates $X,Y$ and $a_{m_{a+3},n_{a+3}}(q)$ are specific to the choice of $\si_0$ and $\tau_0$ (coming from the position (phase) of the Aganagic-Vafa brane). The parameters $\log a_{m_{a+3},n_{a+3}}(q)\to \infty$ for $a=1,\dots, \fp'$ -- heuristically, in the toric graph they measure ``distances'' from other $3$-cones to $\si_0$, while $q_{\fp'+1},\dots,q_\fp$ are parameters for the orbifolds twisted sectors.

One can write the non-equivariant superpotential as the following form
\begin{align*}
&W=H_q(X,Y)Z,\\
&H_q(X,Y)=X^\fr Y^{-\fs}+Y^\fm+1+\sum_{a=1}^\fp a_{m_{3+a},n_{3+a}}(q) X^{m_{3+a}}Y^{n_{3+a}}.
\end{align*}

The equivariantly-perturbed superpotential is
\[
W^{\bT'}=W+\hat x,
\]
where $\hat x=\su_1 x + \su_2 y$. It is a holomorphic function defined on the universal cover $\widetilde \cY=\bC^3$ of $\cY$.\footnote{On the B-model side, we regard $\su_1$ and $\su_2$ as complex numbers.}

\subsubsection{Mirror curve}

The mirror curve $C_q\subset (\bC^*)^2$ is defined by the equation $H_q(X,Y)=0$. The defining polytope $P$ defines a polarized $2$-dimensional toric surface $\bS_P$ with an ample line bundle $L_P$, and $H_q(X,Y)$ extends to a section in $H^0(\bS_P;L_P)$. The zero section is the compactified mirror curve $\oC_q\in \bS_P$. It is of genus $\fg$, and intersects transversally with $\partial \bS_P$ at $\fn$ points (see Section \ref{sec:CR} for the definition of $\fg, \fn$).

In fact, there is an explicit construction of a flat family of toric surfaces over a neighborhood of $q=0$ in \cite{FLZ3}. Each generic fiber is a toric surface isomorphic to $\bS_P$, while the central fiber is $\cup_{\si\in \Si(3)} \bS_{P_\si}$, a normal crossing union of several toric surfaces -- each corresponds to the polytope $P_\si$ in the triangulation of the defining polytope $P$. The toric surface $\bS_P$ degenerates into $\bigcup_{\si\in\Si(3)} \bS_{P_\si}$ at $q=0$. The family $\fC$ of mirror curves is the zero section of a fiberwise ample line bundle -- on each non-degenerate fiber this line bundle restricts to $L_P$. At a generic point $q$, the fiber $\fC_{q}$ is just $\oC_{q}$. At $q=0$, $\fC_{q}$ degenerates into a nodal curve while each piece lives inside $\bS_{P_\si}$.   We denote this neighborhood of $q=0$ by $\mathfrak B$, and denote $$\mathfrak B^\circ=\mathfrak B\cap \left ( \bigcap_{a=1}^\fp \{q_a\neq 0\}\right).$$
Denote $\bD^\infty= \fC\cap (\partial \bS_P)$. This is the family of punctures $\oC_{q}\setminus C_{q}$

\begin{example}[Example \ref{exp:graph}, continued]
  \label{exp:b-model}
  \begin{align*}
    &W^{\bT'}=H_q(X,Y)+\su_1 x + \su_2 y,\\
    &\fr=1,\ \fs=0,\ \fm=2,\\
    &H_q(X,Y)=X+Y^2+1+q_1 X^2 Y^{-1}+q_2 Y,\\
    &q_{K}=q_1,\ q_\orb=q_2.
  \end{align*}
  The mirror curve is illustrated in Figure \ref{fig:mirror-curve}. It is a fattened tube around the toric graph Figure \ref{fig:graph}. Notice that the gerby leg contributes to two punctures $p_0,p_1$. The degenerated mirror curve $\fC_0$ is illustrated in Figure \ref{fig:degen-mirror-curve}.
\end{example}

\begin{figure}[h]
\begin{center}
\psfrag{p1}{\small $p_0$}
\psfrag{p2}{\small $p_1$}
\includegraphics[scale=0.3]{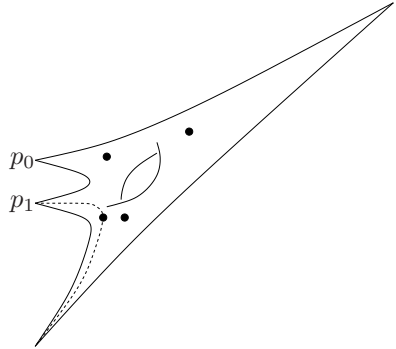}
\caption{The mirror curve of $\cX$. It can be regarded as a curve in the toric surface $\bS_P$. Black dots are the ramification points, while the dashed curve is the Lefschetz thimble passing through one ramification point with $\hat x \to \infty$.}
\label{fig:mirror-curve}
\end{center}
\end{figure}

\begin{figure}[h]
\begin{center}
\psfrag{p1}{\small $p_0$}
\psfrag{p2}{\small $p_1$}
\includegraphics[scale=0.6]{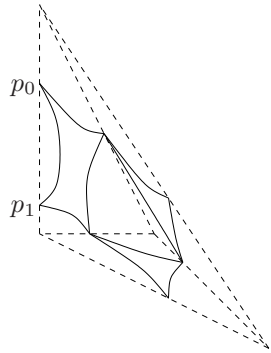}
\caption{The degenerated mirror curve of $\cX$ at $q=0$. It can be regarded as a curve in the degenerated toric surface - a normal crossing of two $\bP^2$ and a $\bP(1,1,2)$, whose moment polytopes are dashed lines.}
\label{fig:degen-mirror-curve}
\end{center}
\end{figure}

\subsubsection{Hori-Vafa mirror}

The Hori-Vafa mirror $(\check \cX_q,\Omega_q)$ is a non-compact Calabi-Yau $3$-fold ($\Omega_q$ is the Calabi-Yau form):
\begin{align*}
    &\check \cX_q=\{(u,v,X,Y)\in \bC^2\times (\bC^*)^2: uv=H_q(X,Y)\},\\
    &\Omega_q=\Res_{\check X_q}\left ( \frac{1}{H(X,Y,q)-uv}du\wedge dv \wedge \frac{dX}{X}\wedge \frac{dY}{Y}\right ).
\end{align*}

\subsection{Open and closed mirror maps from the B-model}

Throughout this section, we choose \emph{framing} $f\in \bZ$ the same as in Section \ref{sec:AV} and let $u_1=1$ and $u_2=f$ in the B-model setting. So
\[
\hat x = x+ fy,\ \hat y = y.
\]
The mirror curve equation for $C_q$ becomes
\[
H(\hat X,\hat Y)=\hat X^\fr \hY^{-\fs-\fr f} + \hY^{\fm} + 1 + \text{other terms.}
\]
The revised coordinates $\hX,\hY$ are specific to the position (``phase'') of the Aganagic-Vafa brane \emph{and} its framing. We denote the Seiberg-Witten form
\[
\Phi=\hat y d\hat x.
\]

\subsubsection{Closed mirror maps as periods}

Let us recall the geometry of the mirror curve $C_q$ and its compactification $\oC_q$ (c.f. Section \ref{sec:CR}).
\begin{itemize}
  \item The genus of $\oC_q$ is $\fg=h^4_\CR(\cX)$.
  \item The number of punctures of $C_q$ is $\fn=\fp-\fg+3$ ($\fp=h^2_\CR(\cX)$, i.e. the number of extended K\"ahler parameters.)
  \item $h_1(\oC_q)=2\fg$, $h_1(C_q)=h_1(\oC_q,\bD^\infty)=\fg+\fn-1=\fp+\fg+2$.
\end{itemize}
The lore for toric Calabi-Yau $3$-folds is that the mirror maps are obtained by integrations
\[
\tau_i=\frac{1}{2\pi\sqrt{-1}}\int_{A_i} \Phi.
\]
Here $A_i$ are $1$-cycles in $H_1(C_q)$, probably with non-integral coefficients in the presence of a toric orbifold. Since $\Phi=\hat y d \hat x$, this map is not well-defined. For practical computation in explicit examples, one can easily make sense of the result by ignoring the constants due to the non-trivial monodromy of $\hat y$. We make this statement a little more precise here.

Consider the inclusion $I:C_{q}\to (\bC^*)^2$ where $q\in \fB^\circ$. Denote the kernel of the map
\[
K_1(C_q;\bZ)=\ker I_*\subset H_1(C_q;\bZ).
\]
This sublattice $K_1(C_q;\bZ)\cong \bZ^{2\fg+\fn-3}$ consists of cycles with trivial $x$ and $y$ monodromy. One can lift any element $\gamma \in K_1(C_q;\bZ)$ to a loop $\tgamma\in\pi_1(\tC_q)$. Here $\tC_q=\pi^{-1}(C_q)$ under the universal covering map $\pi:\bC^2\to (\bC^*)^2$. The following integral
\[
\frac{1}{2\pi\sqrt{-1}}\int_{\tgamma} \Phi\in \bC
\]
does depend on the choice of $\tgamma$, but only up to an integral multiple of $2\pi\sqrt{-1}$. If the toric Calabi-Yau $3$-fold $\cX$ is smooth, then one can choose $A_a\in K^{\inv}_1(C_q;\bZ)$ for $a=1,\dots,\fp$, the \emph{monodromy invariant} cycles,\footnote{Monodromy under the Gauss-Manin connection around $\fB^\circ$, i.e. around hyperplanes $\{q_a=0\}$.} such that
\[
\frac{1}{2\pi\sqrt{-1}}\int_{A_a} \Phi\in \bC=\log q_a+O(q) \mod 2\pi\sqrt{-1}.
\]
Since for each $A\in K^\inv_1(C_q;\bZ)$ can be extended a family of \emph{flat} cycles over $\fB^\circ$, we have a map
\[
A\mapsto \frac{1}{2\pi\sqrt{-1}}\int_A \Phi\in \Hol (\fB^\circ, \bC/(2\pi\sqrt{-1}\bZ)).
\]
Here $\Hol$ is the set of holomorphic maps (in paricular $\log q_a$ is in this set). More subtleties arise when $\cX$ is an orbifold. To acquire the desired leading order behavior, we have to use cycles in $K_1^\inv(C_q;\bC)$. The explicit construction of cycles $A_1,\dots,A_\fp$ is in \cite{FLZ3}. We list the result here
\begin{proposition}
  \label{prop:A-integral}
  There exists cycles $A_1,\dots,A_{\fp'},\dots,A_\fp\in K^\inv_1(C_q;\bC)$ such that
  \[
  \frac{1}{2\pi\sqrt{-1}}\int_{A_a} \Phi=
  \begin{cases}
     \log q_a +O(q_K)+O(|q_\orb|^2),\ {a=1,\dots,\fp'}\\
     q_a+O(q_K)+O(|q_\orb|^2),\ {a=\fp',\dots,\fp}.
  \end{cases}.
  \]
\end{proposition}
One has to understand the right hand side of this proposition as holomorphic functions in $q_1,\dots,q_\fp$ up to constants.

\subsubsection{Large radius limit for the open parameter and open mirror map}

A point on $C_q$ heuristically correponds to a \emph{B-brane}. The large radius limit of the open parameter should correspond to moving the Aganagic-Vafa brane $\cL$ to infinity along the leg. We define the \emph{large radius points} $\bar p_0(q),\dots,\bar p_{\fm-1}(q)$ on $\fC_q$ by requiring $\hX=0$ and
\[
\hY^\fm=-1,\ \text{when }q=0.
\]
We also denote $D_q^\ell$ to be a small neighborhood of $\bar p_\ell(q)$ on $\oC_{q}=\fC_q$.

There is a $\bmu^*_\fm$-action permuting these points $\bar p_\ell$, given by mapping the $\hat Y$-coordinate $(-1)^{\frac{k}{\fm}}$ to $(-1)^{\frac{1}{\fm}} \chi(1^\frac{k}{\fm})$ for $\chi\in \bmu_\fm^*$. We can identify these points $\bar p_\ell$, $\ell=1,\dots,\fm$ with elements in $\bmu^*_\fm$ in a non-canonical way. Define
$$
\psi_\ell:= \frac{1}{\fm} \sum_{k=0}^{\fm-1}\omega_\fm^{-k\ell} \one'_{\frac{k}{\fm}}, \quad \ell=0,1, \ldots, \fm-1,
$$
where $\omega_\fm= e^{2\pi\sqrt{-1}/\fm}$. Then $\psi_\ell$ for $\ell\in \bmu_\fm^*$ form a dual basis to $\one'_{\frac{k}{m}}$ for $k\in \bmu_\fm$.

As indicated in \cite{AKV}, there is an explicit constuction of a (linear combination of) path $\tgamma_0$ in $\tC_q$. Each component in $\tgamma_0$ is a path starting at the point where $\hat x$-coordinate is $\hat x_0$ near $\bar p_\ell$,  such that
\begin{itemize}
  \item $\tgamma_0$ descends to a cycle $A_0$ in $H_1(C_q;\bC)$ -- it has trivial monodromy in $\hat x$ while the monodromy in $\hat y$ is precisely $\hat y\mapsto \hat y-2\pi\sqrt{-1}$;
  \item The integration depends on the starting point:
  \begin{align}\frac{1}{2\pi\sqrt{-1}}\int_{\tgamma_0} \hat y d\hat x = \hat x_0 +  O(q_K)+O(|q_\orb|^2).\label{eqn:open-mirror-map-integral}\end{align}
    This will be the open mirror \eqref{eqn:open-mirror-map}.
\end{itemize}

\section{A quick review of the genus zero mirror theorem for toric orbifolds}
\label{sec:genus-0-mirror}

\subsection{Frobenius structures for toric Calabi-Yau 3-orbifolds: quantum cohomology}
Let
$$
\chi= \dim_{\bC}H^*_{\CR}(\cX) =\dim_{\bST} H^*_{\CR,\bT}(\cX;\bST).
$$
We choose a $\bST$-basis of $H^*_{\CR,\bT}(\cX;\bST)$ $\{ T_i: i=0,1,\ldots,\chi-1\}$
such that
$$
T_0=1,\quad T_a=\bar{H}_{3+a}^{\bT} \textup{ for } a=1,\ldots,\fp',\quad
T_a=\one_{b_{3+a}}\textup{  for }a=\fp'+1,\ldots, \fp,
$$
and for $i=\fp+1,\ldots,\chi-1$, $T_i$ is of the form $T_aT_b$ for some $a,b\in \{1,\ldots, \fp\}$.
Write $t=\sum_{a=0}^{\chi-1}\tau^a T_a$, and let $\tau'=(\tau_1,\ldots, \tau_{\fp'})$,
$\tau''=(\tau_0, \tau_{\fp'+1},\ldots, \tau_{\chi-1})$.
By the divisor equation,
$$
\llangle T_i,T_j,T_k\rrangle^{\cX,\bT}_{0,3} \in \ST[\![ \tQ, \tau'']\!],\quad
\llangle \hat{\phi}_{\bsi}, \hat{\phi}_{\bsi'}, \hat{\phi}_{\bsi''}\rrangle^{\cX,\bT}_{0,3}\in \bSTQ,
$$
where $\tQ^d = Q^d \exp(\sum_{a=1}^{\fp'}\tau_a \langle T_a, d\rangle)$. Let $S:= \bSTQ$. Given
$a,b\in H_{\CR,\bT}^*(\cX;\bST)$, define the \emph{quantum product}
$$
a\star_t b:= \sum_{\bsi\in I_\Si} \llangle a, b, \hat{\phi}_\bsi\rrangle \hat{\phi}_\bsi
\in H_{\CR,\bT}^*(\cX;\bST)\otimes_{\bST} S.
$$
Then $A:= H_{\CR,\bT}^*(\cX;\bST)\otimes_{\bST}S$ is a free $S$-module of rank $\chi$, and
$(A,*_t)$ is a commutative, associative algebra over $S$.  Let $I\subset S$
be the ideal generated by $\tQ,\tau''$, and define
$$
S_n:= S/I^n,\quad A_n:= A\otimes_{S}S_n
$$
for $n\in \bZ_{\geq 0}$. Then $A_n$ is a free $S_n$-module of rank $\chi$, and the ring structure
$*_t$ on $A$ induces a ring structure $*_{\underline{n}}$ on $A_n$. In particular,
$$
S_1=\bST,\quad A_1=H^*_{\CR,\bT}(\cX;\bST),
$$
and $*_{\underline{1}} = *_\cX$ is the orbifold cup product. So
$$
\{ \phi_{\bsi}^{(1)}:=\phi_{\bsi}: \bsi\in I_\Si\}
$$
is an idempotent basis of $(A_1,\star_{\underline{1}})$.
For $n\geq 1$, let  $\{ \phi_{\bsi}^{(n+1)}:\bsi\in I_\Si\}$
be the unique idempotent basis of $(A_{n+1},\star_{\underline{n+1}})$
which is the lift of the idempotent basis $\{ \phi_{\bsi}^{(n)}:\bsi\in I_\Si\}$
of $(A_n,\star_{\underline{n}})$ \cite[Lemma 16]{LP}. Then
$$
\{ \phi_{\bsi}(t):=\lim \phi_{\bsi}^{(n)}: \bsi\in I_\Si\}
$$
is an idempotent basis of $(A,\star_t)$. The ring $(A,\star_t)$ is called the \emph{equivariant big quantum cohomology ring}, which is also denoted by $QH^*_{\bT}(\cX)$.

Set
$$
\novT:= \bST\otimes_{\bC}\nov =\bST[\![ E(\cX)]\!].
$$
Then $H:=H^*_{\CR,\bT}(\cX;\novT)$ is a free $\novT$-module of rank $\chi$.
Any point $t\in H$ can be written as  $t=\sum_{\bsi\in I_\Si}t^{\bsi} \hat{\phi}_{\bsi}$
. We have
$$
H=\mathrm{Spec}(\novT [ t^{\bsi}:\bsi\in I_\Si]).
$$
Let $\hat{H}$ be the formal completion of $H$ along the origin:
$$
\hat{H} :=\mathrm{Spec}(\novT[\![ t^{\bsi}:\bsi\in I_\Si ]\!]).
$$
Let $\cO_{\hat{H}}$ be the structure sheaf on $\hat{H}$, and let $\cT_{\hat{H}}$ be the tangent sheaf on
$\hat{H}$.
Then $\cT_{\hat{H}}$ is a sheaf of free $\cO_{\hat{H}}$-modules of rank $\chi$.
Given an open set in $\hat{H}$,
$$
\cT_{\hat{H}}(U)  \cong \bigoplus_{\bsi\in I_\Si}\cO_{\hat{H}}(U) \frac{\partial}{\partial t^{\bsi}}.
$$
The big quantum product and the $\bT$-equivariant Poincar\'{e} pairing defines the structure of a formal
Frobenius manifold on $\hat{H}$:
$$
\frac{\partial}{\partial t^{\bsi}} \star_t \frac{\partial}{\partial t^{\bsi'}}
=\sum_{\brho\in I_\Si} \llangle \hat{\phi}_{\bsi},\hat{\phi}_{\bsi'},\hat{\phi}_{\brho}\rrangle_{0,3}^{\cX,\bT}
\frac{\partial}{\partial t^{\brho}}
\in \Gamma(\hat{H}, \cT_{\hat{H}}).
$$
$$
( \frac{\partial}{\partial t^{\bsi}},\frac{\partial}{\partial t^{\bsi'}})_{\cX,\bT} =\delta_{\bsi,\bsi'}.
$$
The length of the canonical basis in the equivariant Chen-Ruan cohomology and the equivariant quantum cohomology are denoted as
\[
(\phi_\bsi,\phi_\bsi)_{\cX,\bT}=\frac{1}{\Delta^\bsi},\quad (\phi_\bsi(t),\phi_\bsi(t))_{\cX,\bT}=\frac{1}{\Delta^\bsi(t)}.
\]

By replacing $\bT$ by $\bT'$ in the above discussion, we obtain the equivariant big quantum cohomology ring $QH^*_{\bT'}(\cX)$.

\subsection{The B-model Frobenius structure: the Jacobian ring}

We can define a Frobenius algebra for each $q$:
\[
\mathrm {Jac}(W^{\bT'})=\frac{\bC[X^\pm,Y^\pm,Z^\pm]}{\langle \frac{\partial W^{\bT'}}{\partial X}, \frac{\partial W^{\bT'}}{\partial Y}, \frac{\partial W^{\bT'}}{\partial Z} \rangle }.
\]
The ring structure is self-evident in the definition, while the metric is
\[
(f,g)=\frac{1}{(2\pi\sqrt{-1})^3}\int_{|dW^{\bT'}=\epsilon|}\frac{fgdx\wedge dy\wedge dz}{\frac{\partial W^{\bT'}}{\partial x}\frac{\partial W^{\bT'}}{\partial y}\frac{\partial W^{\bT'}}{\partial z}}=\sum_{\balpha} \frac{f(P_\balpha)g(P_\balpha)}{\det\Hess_{P_\balpha}(W^{\bT'})}.
\]
In this expression $P_\balpha$ runs through all critical points of $W^{\bT'}$. An element in $\mathrm{Jac}(W^{\bT'})$ is a Laurent polynomial in $\bC[X^\pm,Y^\pm,Z^\pm]$ while two Laurent polynomials are identified if their values on critical points of $W^{\bT'}$ are the same. The canonical basis of $\mathrm {Jac}(W^{\bT'})$ is a function taking value $1$ on one critical point of $W^{\bT'}$ while vanishing at all other critical points.

Notice the current set-up does not constitute a Frobenius \emph{manifold} -- we just have a family of Frobenius algebras when varying $q$.

\subsection{Genus $0$ mirror symmetry for toric Calabi-Yau 3-orbifolds: identification of $I-$function and $J$-function}
\label{sec:genus-0-mirror-theorem}
Given the choice of $H_1,\dots, H_\fp$ (and the lift $\bar H_1^\bT,\dots, \bar H_{\fp'}^\bT\in H^2_{\bT}(\cX)\subset H^2_{\CR,\bT}(\cX)$), for any $\beta\in \bK_\eff$, we define
\[
q^\beta=\prod_{a=1}^\fp q_a^{\langle H_a,\beta\rangle}.
\]
It is a monomial of $q_1,\dots,q_\fp$. We define $\bT$-equivariant small $I$-function as follows. Recall that $q_K=(q_1,\dots,q_{\fp'})$ for the K\"ahler part, while $q_\orb=(q_{\fp'+1},\dots,q_{\fp})$ for the twisted sector.
\begin{definition}\label{def-I}
\begin{eqnarray*}
I_{\bT}(t_0,q,z) &=& e^{ (\sum_{a=1}^{\fp'} \bar H_a^\bT \log q_a)/z}
\sum_{\beta\in \bK_\eff} q^\beta \prod_{i=1}^{3+\fp'}
\frac{\prod_{m=\lceil \langle D_i, \beta \rangle \rceil}^\infty
(\bar{D}^\bT_i  +(\langle D_i,\beta\rangle -m)z)}{\prod_{m=0}^\infty(\bar{D}^\bT_i+(\langle D_i,\beta\rangle -m)z)} \\
&& \quad \cdot \prod_{i=4+\fp'}^{3+\fp} \frac{\prod_{m=\lceil \langle D_i, \beta \rangle \rceil}^\infty (\langle D_i,\beta\rangle -m)z}
{\prod_{m=0}^\infty(\langle D_i,\beta\rangle -m)z} \one_{v(\beta)}.
\end{eqnarray*}
\end{definition}

\begin{definition}
  The equivariant $J$-function is
  \[
  J_{\bT}(\btau,z)=\sum_{\bsi\in I_{\Si}}\llangle  \frac{\phi_\bsi}{z-\psi_1},1\rrangle_{0,2}^\cX{\phi^\bsi}.
  \]
\end{definition}

The main result in \cite{CCIT} implies the following $\bT$-equivariant
mirror theorem:
\begin{theorem}[Coates-Corti-Iritani-Tseng]\label{thm:IJ}
$$
J_{\bT}(\btau,z)\vert_{Q=1} = I_{\bT}(t_0,q, z),
$$
where the equivariant closed mirror map  $q\mapsto  \btau(q)\in H^2_{\CR,
\bT}(\cX)$
is determined by the first-order term in the asymptotic expansion of the $I$-function
$$
I(t_0,q,z)=1+\frac{ \btau(q)}{z}+o(z^{-1}).
$$
More explicitly, the equivariant closed mirror map is given by
$$
\btau =\tau_0(q) + \sum_{a=1}^{\fp'}\tau_a(q) \bar{H}^\bT_a +\sum_{a=\fp'+1}^\fp \tau_a(q) \one_{b_{a+3}},
$$
where each $\tau_a(q)$ can be obtained readily from the expansion of the $I$-function. They have the following asymptotic behavior:
\begin{eqnarray}
\label{eqn:closed-mirror-map}
\tau_0(q) &=& O(q)\in \sum_{i=1}^3 \bC \su_i, \nonumber\\
\tau_a(q) &=& \begin{cases}
\log(q_a)+ O(q)
, & 1\leq a\leq \fp',\\
q_a+\text{higher order terms}, & \fp'+1\leq a\leq \fp.
\end{cases}
\end{eqnarray}
\end{theorem}

It does make sense to set $Q=1$ -- the Nokikov variable could eventually be dropped, since in principle by the divisor equation $e^{\tau}$ carries the same information as $Q$. See Remark \ref{Novikov} for a precise argument.

Under this mirror map, the B-model large radius/orbifold mixed-type
limit $q\to 0$ corresponds to the A-model large radius/orbifold
mixed type limit $\tau'\to -\infty, \tau''\to 0$.

The mirror maps in Equation \eqref{eqn:closed-mirror-map} and the period integrals in Proposition \ref{prop:A-integral} are solutions to certain GKZ system with prescribed asymptotic behavior. They have same leading order, and this ensures that they are equal. The closed mirror map $\tau_i$ is the period integrals along $A_i$ cycle.
\begin{proposition}
  In the mirror map Equation \eqref{eqn:closed-mirror-map}
  \[
  \tau_a=\frac{1}{2\pi\sqrt{-1}}\int_{A_a}\Phi \mod \const.
  \]
\end{proposition}

\subsection{Genus $0$ mirror symmetry for toric Calabi-Yau 3-orbifolds: identification of Frobenius algebras}

\label{sec:identification-frobenius}

 The genus $0$ mirror theorem of \cite{CCIT} implies the following under the closed mirror map $\btau=\btau(q)$
\begin{equation}
  \label{eqn:identification-frobenius}
\mathrm{Jac}(W^{\bT'})\cong QH^*_{\bT'}(\cX)\vert_{Q=1}.
\end{equation}
This statement should be understood a pointwise isomorphism of Frobenius algebras.\footnote{A reasonably enlarged B-model with more parameters than $q$ should produce a Frobenius manifold isomorphic to $QH^*_{\bT'}$. However we only need pointwise isomorphism of Frobenius algebras for the purpose of proving the Remodeling Conjecture.} This statement identifies $\mathrm{Jac}(W^{\bT'})$ with a \emph{slice} of $QH^*_{\bT'}(\cX)$. From this fact, there is an identification of canonical basis for $\mathrm{Jac}(W^{\bT'})$ and $QH^*_{\bT'}(\cX)$ when $\btau=\btau(q)$. So the index set of the canonical basis $I_\Si$ is identified with the index set of critical points of $W^{\bT'}$, which we also denote by $I_\Si$. We use bold greek letters like $\balpha$ and $\bsi$ to denote such an index. The dimension $N=\dim QH_{\bT'}^*(\cX)$ and  \emph{it is also the number of critical points of $W^{\bT'}$.}

Each critical point $P_\balpha$ of $W^{\bT'}$ in $(\bC^*)^3$ (although $W^{\bT'}$ is defined on the universal cover of $(\bC^*)^3$, $dW^{\bT'}$ is well-defined on $(\bC^*)^3$), by direct calculation, is characterized by the following
\begin{align*}
&\text{$P_\balpha=(X_\balpha,Y_\balpha,Z_\balpha(X_\balpha,Y_\balpha))$ is a critical point of $W^{\bT'}$}\iff \\
&\text{$p_\balpha=(X_\balpha, Y_\balpha)$ is a critical point of $d\hat x$ on the mirror curve $C_q$.}
\end{align*}
Let's denote the canonical basis $V_\balpha\in \Jac(W^{\bT'})$ to be the function taking value $1$ on $P_\balpha$ and $0$ on $P_\bbeta,\bbeta\neq \balpha$. In particular, $W^{\bT'}(P_\balpha)=\hat x(p_\balpha)$, i.e. the \emph{canonical coordinate} for each $V_\balpha$ is the \emph{critical value} of $W^{\bT'}$, which is also the branch value of $\hat x_{0,\balpha}:=\hat x(p_\balpha)$. A straightforward calculation shows
\begin{itemize}
  \item One sets $\hat x = \hat x_{0,\balpha}+\zeta_\alpha^2$. In the expansion $\hat y = \hat y_{0,\balpha} +\sum_{d=1}^\infty h^\balpha_d \zeta_\alpha^d$, $$h_1^\balpha=\sqrt{\frac{2}{\frac{d^2 \hat x}{d \hat y^2}}};$$
  \item $\det \Hess_{P_\balpha}(W^{\bT'})=-\frac{d^2 \hat x}{d\hat y^2}(p_\balpha);$
  \item The squared norm of the canonical basis $$(V_\balpha,V_\balpha)= \frac{1}{\det\Hess_{P_\balpha}(W^{\bT'})}.$$
\end{itemize}
Since we denote $(\hat \phi_\balpha(\btau),\hat\phi_\balpha(\btau))=\frac{1}{\Delta^\balpha(\btau)}$, by the isomorphism of Frobenius algebra under the closed mirror map $\btau=\btau(q)$, we conclude that
\begin{equation}
\frac{h^\balpha_1}{\sqrt{-2}}=\frac{1}{\sqrt{\Delta^\balpha(\btau)}}.
\label{eqn:length-matching}
\end{equation}
We will see in Section \ref{sec:vertices} this identifies the vertex terms of A and B-model graph sum formulae (Theorem \ref{thm:Zong} and \ref{thm:DOSS}).

\section{A-model quantization: the orbifold Givental formula}

\label{sec:A-model-quantization}

\subsection{The equivariant big quantum differential equation}
We consider the Dubrovin connection $\nabla^z$, which is a family
of connections parametrized by $z\in \bC\cup \{\infty\}$, on the tangent bundle
$T_{\hat{H}}$ of the formal Frobenius manifold $\hat{H}$:
$$
\nabla^z_{\bsi}=\frac{\partial}{\partial t^{\bsi}} -\frac{1}{z} \hat{\phi}_{\bsi}\star_t
$$
The commutativity (resp. associativity)
of $*_t$ implies that $\nabla^z$ is a torsion
free (resp. flat) connection on $T_{\hat{H}}$ for all $z$. The equation
\begin{equation}\label{eqn:qde}
\nabla^z \mu=0
\end{equation}
for a section $\mu\in \Gamma(\hat{H},\cT_{\hat{H}})$ is called the {\em $\bT$-equivariant
big quantum differential equation} ($\bT$-equivariant big QDE). Let
$$
\cT_{\hat{H}}^{f,z}\subset \cT_{\hat{H}}
$$
be the subsheaf of flat sections with respect to the connection $\nabla^z$.
For each $z$, $\cT_{\hat{H}}^{f,z}$ is a sheaf of
$\novT$-modules of rank $\chi$.

A section $L\in \End(T_{\hat{H}})=\Gamma(\hat{H},\cT_{\hat{H}}^*\otimes\cT_{\hat{H}})$
defines an $\cO_{\hat{H}}(\hat{H})$-linear map
$$
L: \Gamma(\hat{H},\cT_{\hat{H}})= \bigoplus_{\bsi\in I_{\Si}} \cO_{\hat{H}}(\hat{H})
\frac{\partial}{\partial t^{\bsi}}
\to \Gamma(\hat{H},\cT_{\hat{H}})
$$
from the free $\cO_{\hat{H}}(\hat{H})$-module $\Gamma(\hat{H},\cT_{\hat{H}})$ to itself.
Let $L(z)\in \End(T_{\hat H})$ be a family of endomorphisms of the tangent bundle $T_{\hat{H}}$
parametrized by $z$. $L(z)$ is called a {\em fundamental solution} to the $\bT$-equivariant QDE if
the $\cO_{\hat{H}}(\hat{H})$-linear map
$$
L(z): \Gamma(\hat{H},\cT_{\hat{H}}) \to \Gamma(\hat{H},\cT_{\hat{H}})
$$
restricts to a $\novT$-linear isomorphism
$$
L(z): \Gamma(\hat{H},\cT_H^{f,\infty})=\bigoplus_{\bsi\in I_{\Si}} \novT \frac{\partial}{\partial t^{\bsi}}
\to \Gamma(\hat{H},\cT_H^{f,z}).
$$
between rank $\chi$ free $\novT$-modules.

\subsection{The $\cS$-operator}\label{sec:A-S}
The $\cS$-operator is defined as follows.
For any cohomology classes $a,b\in H_{\CR,\bT}^*(\cX;\bST)$,
$$
(a,\cS(b))_{\cX,\bT}=(a,b)_{\cX,\bT}
+\llangle a,\frac{b}{z-{\psi}}\rrangle^{\cX,\bT}_{0,2}
$$
where
$$
\frac{b}{z-{\psi}}=\sum_{i=0}^\infty b{\psi}^i z^{-i-1}.
$$
The $\cS$-operator can be viewed as an element in $\End(T_{\hat{H}})$ and is a fundamental solution to the $\bT$-equivariant
big QDE \eqref{eqn:qde}.  The proof for $\cS$ being a fundamental solution can be found in \cite{CK}
for the smooth case and in \cite{Ir09} for the orbifold case.

\begin{remark}
One may notice that since there is a formal variable $z$ in the definition of
the $\bT$-equivariant big QDE \eqref{eqn:qde}, one can consider its solution space over different rings. Here the operator
$\cS= \one+ \cS_1/z+ \cS_2/z^2+\cdots$ is viewed as a formal power series in $1/z$ with operator-valued coefficients.
\end{remark}

\begin{remark}\label{Novikov}
Given $t\in H^*_{\CR,\bT}(\cX)\otimes_\RT \bST$, let $t=t'+t''$ where $t'\in H^2_{\bT}(\cX)\otimes_\RT \bST$ and $t''$ is a linear combination of elements in $H^{\neq 2}_{\CR,\bT}(\cX)\otimes_\RT \bST$ and elements in degree 2 twisted sectors. Then by divisor equation, we have
$$
(a,b)_{\cX,\bT}+\llangle a,\frac{b}{z-{\psi}}\rrangle^{\cX,\bT}_{0,2}=(a,be^{t'/z})_{\cX,\bT}+
\sum_{m=0}^\infty \sum_{d\in E(\cX)\\(d,m)\neq(0,0)}\frac{Q^de^{\int_dt'}}{m!}\langle a,\frac{be^{t'/z}}{z-{\psi}},(t'')^m\rangle^{\cX,\bT}_{0,2+m,d}.
$$
In the above expression, if we fix the power of $z^{-1}$, then only finitely many terms in the expansion of $e^{t'/z}$ contribute. Therefore, the factor $e^{\int_dt'}$ can play the role of $Q^d$ and hence the restriction $\llangle a,\frac{b}{z-{\psi}}\rrangle^{\cX,\bT}_{0,2}|_{Q=1}$ is well-defined. So the operator $\cS|_{Q=1}$ is well-defined.
\end{remark}

We consider several different (flat) bases for $H_{\CR,\bT}^*(\cX;\bST)$:
\begin{enumerate}
\item The classical canonical basis $\{ \phi_{\bsi}:\bsi\in I_\Si \}$ defined in Section \ref{sec:equivariantCR}.
\item The basis dual to the classical canonical basis with respect to the $\bT$-equivariant Poincare pairing:
$\{ \phi^{\bsi} =\Delta^{\bsi} \phi_{\bsi}: \bsi \in I_\Si \} $, where $1/\Delta^\bsi=(\phi_\bsi,\phi_\bsi)$.
\item The classical normalized canonical basis
$\{ \hat{\phi}_{\bsi}=\sqrt{\Delta^{\bsi}}\phi_{\bsi} :\bsi\in I_\Si\}$ which is self-dual: $\{ \hat{\phi}^{\bsi}=\hat{\phi}_{\bsi}: \bsi \in I_\Si \}$.
\end{enumerate}

We also consider several different non-flat basis:
\begin{enumerate}
  \item The quantum canonical basis $\{ \phi_{\bsi}(t):\bsi\in I_\Si \}$, such that at LRL $\lim_{\tQ\to 0,t''\to 0}\phi_{\bsi}(t)=\phi_{\bsi}$.
  \item The basis dual to the qunatum canonical basis with respect to the $\bT$-equivariant Poincare pairing:
  $\{ \phi^{\bsi} (t) =\Delta^{\bsi}(t) \phi_{\bsi}(t): \bsi \in I_\Si \} $, where $(\phi_\bsi(t),\phi_\bsi(t))=1/\Delta^\bsi(t)$.
  \item The quantum normalized canonical basis
  $\{ \hat{\phi}_{\bsi}(t)=\sqrt{\Delta^{\bsi}(t)}\phi_{\bsi}(t) :\bsi\in I_\Si\}$ which is self-dual:  $\hat{\phi}^{\bsi}(t)=\hat{\phi}_{\bsi}(t): \bsi \in I_\Si$.
\end{enumerate}
The ``classical'' basis are flat while the ``quantum'' ones are not. For most of our application, we will set $Q=1$, and $t=\btau$ for an $H^{\leq 2}$ element.

For $\bsi, \bsi'\in I_\Si$, define
$$
S^{\bsi'}_{\spa \bsi}(z) := (\phi^{\bsi}, \cS(\phi_{\bsi})).
$$
Then $(S^{\bsi'}_{\spa \bsi}(z))$ is the matrix  of the $\cS$-operator with respect to the canonical basis
$\{\phi_{\bsi}:\bsi\in I_\Si \}$:
\begin{equation}\label{eqn:S}
\cS(\phi_{\bsi}) =\sum_{\bsi'\in I_\Si}
\phi_{\bsi'} S^{\bsi'}_{\spa \bsi}(z).
\end{equation}

For $\bsi,\bsi'\in I_\Si$, define
$$
S_{\bsi'}^{\spa \widehat{\bsi} }(z) := (\phi_{\bsi'}, \cS(\hat{\phi}^{\bsi})).
$$
Then $(S_{\bsi'}^{\spa  \widehat{\bsi}})$ is the matrix of the $\cS$-operator
with respect to the bases $\{\hat{\phi}^{\bsi}:\bsi\in I_\Si\}$ and
$\{\phi^{\bsi}: \bsi\in I_{\Si}\}$:
\begin{equation}\label{eqn:barS}
\cS(\hat{\phi}^{\bsi})=\sum_{\bsi'\in I_{\Si}} \phi^{\bsi'}
 S_{\bsi'}^{\spa \widehat{\bsi}}(z).
\end{equation}

Introduce
\begin{align*}
S_z(a,b)&=(a,\cS(b))_{\cX,\bT},\\
V_{z_1,z_2}(a,b)&=\frac{(a,b)_{\cX,\bT}}{z_1+z_2}+\llangle \frac{a}{z_1-\psi_1},
                  \frac{b}{z_2-\psi_2}\rrangle^{\cX,\bT}_{0,2}.
\end{align*}
A well-known WDVV-like argument says
\begin{equation}
\label{eqn:two-in-one}
V_{z_1,z_2}(a,b)=\frac{1}{z_1+z_2}\sum_i S_{z_1}(T_i,a)S_{z_2}(T^i,b),
\end{equation}
where $T_i$ is any basis of $H^*_{\CR,\bT}(\cX;\bST)$ and $T^i$ is its dual basis.
In particular,
$$
V_{z_1,z_2}(a,b)=\frac{1}{z_1+z_2}\sum_{\bsi\in I_\Si} S_{z_1}(\hat{\phi}_{\bsi},a)S_{z_2}(\hat{\phi}_{\bsi},b).
$$

\subsection{Quantization of quadratic Hamiltonians}
In this section, we review the basic concepts of the quantization of quadratic Hamiltonians (see \cite{Gi01'} for more details). The quantization procedure provides a way to recover the higher genus theory from the genus zero data which we will use in the next section.

\subsubsection{Symplectic space formalism}
So far, we have been working on (a formal neighborhood of) the space $H=\mathrm{Spec}(\novT [ t^{\bsi}:\bsi\in I_\Si])$ which provides us the Frobenius structure and state space of the corresponding Gromov-Witten theory. When we consider the descendent theory of $\cX$, however, additional parameters are needed. Let $\bt(\psi)=t_0+t_1\psi+t_2\psi^2+\cdots$ be a formal power series in $\psi$ with an integer index that keeps track in the power of $\psi$. Here each $t_a$ lies in $H^*_{\CR,\bT'}(\cX)$. We define
$$\langle \bt(\psi_1),\cdots,\bt(\psi_k)\rangle_{g,k,\beta}^{\cX,\bT}=
\int_{[\Mbar_{g,k}(\cX,\beta)^{\bT}]^\vir}
\frac{\prod_{j=1}^{k}(\sum_{a=0}^{\infty}(\ev_j^*t_a)\psi_j^a)}{e_{\bT}(N^{\vir})}.$$
The additional index $a$ leads to the study of the symplectic space formalism.

Let $z$ be a formal variable. We consider the space $\bH$ which is the space of Laurent polynomials in one variable $z$ with coefficients in $H$. We define the symplectic form $\Omega$ on $\bH$ by
$$\Omega(f,g)=\Res_{z=0}( f(-z),g(z)))_{\cX,\bT}dz$$
for any $f,g\in\bH$. Note that we have $\Omega(f,g)=-\Omega(g,f)$. There is a natural polarization $\bH=\bH_+\oplus \bH_-$ corresponding to the decomposition $f(z,z^{-1})=f_+(z)+f_-(z^{-1})z^{-1}$ of laurent polynomials into polynomial and polar parts. It is easy to see that $\bH_+$ and $\bH_-$ are both Lagrangian subspaces of $\bH$ with respect to $\Omega$.

Introduce a Darboux coordinate system $\{p^\bsi_a,q^\brho_b\}$ on $\bH$ with respect to the above polarization. This means that we write a general element $f\in\bH$ in the form
$$\sum_{a\geq 0,\bsi\in I_\Sigma}p^\bsi_a\hat{\phi}^\bsi(-z)^{-a-1}+\sum_{b\geq 0,\brho\in I_\Sigma}q^\brho_b\hat{\phi}_\brho z^b.$$
Denote
\begin{eqnarray*}
\textbf{p}(z):&=&p_0(-z)^{-1}+p_1(-z)^{-2}+\cdots\\
\textbf{q}(z):&=&q_0z+q_1z^2+\cdots,
\end{eqnarray*}
where $p_a=\sum_\bsi p^\bsi_a \hat{\phi}^\bsi$ and $q_b=\sum_\brho q^\brho_b\hat{\phi}_\brho$.

Recall that when we discussed the Gromov-Witten theory of $\cX$, we introduced the formal power series $\bt(z)=t_0+t_1z+t_2z^2+\cdots$. With $z$ replaced by $\psi$, $\bt$ appears as the insertion in the genus $g$ correlator. We relate $\bt(z)$ to the Darboux coordinates by introducing the \emph{dilaton shift}: $\textbf{q}(z)=\bt(z)-\one z$. The dilaton shift appears naturally in the quantization procedure. We will explain this phenomenon as a group action on Cohomological field theories in the next section.

\subsubsection{Quantization of quadratic Hamiltonians}
Let $A:\bH\to\bH$ be a linear infinitesimally symplectic transformation, i.e. $\Omega(Af,g)+\Omega(f,Ag)=0$ for any $f,g\in\bH$. Under the Darboux coordinates, the quadratic Hamiltonian
$$f\to\frac{1}{2}\Omega(Af,f)$$
is a series of homogeneous degree two monomials in $\{p^\bsi_a,q^\brho_b\}$. Let $\hbar$ be a formal variable and define the quantization of quadratic monomials as
$$\widehat{q^\bsi_aq_b^\brho}=\frac{q^\bsi_aq_b^\brho}{\hbar},\widehat{q^\bsi_ap_b^\brho}=
q^\bsi_a\frac{\partial}{\partial q^\brho_b},
  \widehat{p^\bsi_ap_b^\brho}=\hbar \frac{\partial}{\partial q^\bsi_a}\frac{\partial}{\partial q^\brho_b}.
$$
We define the quantization $\widehat{A}$ by extending the above equalities linearly. The differential operators $\widehat{q^\bsi_aq_b^\brho},\widehat{q^\bsi_ap_b^\brho},\widehat{p^\bsi_ap_b^\brho}$ act on the so called Fock space \emph{Fock} which is the space of formal functions in $\bt(z)\in\bH_+$. For example, the descendent potential and ancestor potential are regarded as elements in \emph{Fock}. The quantization operator $\widehat{A}$ does not act on \emph{Fock} in general since it may contain infinitely many monomials. However, the actions of quantization operators in our paper are well-defined. The quantization of a symplectic transform of the form $\exp(A)$, with $A$ infinitesimally symplectic, is defined to be $\exp(\widehat{A})=\sum_{n\geq 0}\frac{\widehat{A}^n}{n!}$.

\begin{remark}
Let $A:\bH\to\bH$ be a linear infinitesimally symplectic transformation. The quantization of $Az^m, m\geq 0$ is studied in \cite[equation (1.3)]{C03} and in \cite[Appendix C]{Ts10}. There is a sign error in the second term of  \cite[equation (1.3)]{C03} and the corresponding identity in \cite[Appendix C]{Ts10}. The sign $(-1)^k$ should be replaced by $(-1)^{k+m-1}$.

\end{remark}

\subsection{Givental's formula}
Let $U$ denote the diagonal matrix whose diagonal entries are the canonical coordinates. So $u^\bsi$ are canonical coordinates for each $\bsi\in I_\Si$.
The results in \cite{Gi01} and \cite{Zo} imply the following statement.
\begin{theorem}\label{R-matrix}
There exists a unique matrix power series $R(z)= \one + R_1z+R_2 z^2+\cdots$
satisfying the following properties.
\begin{enumerate}
\item The entries of $R_d$ lie in $\bSTQ$.
\item $\tS=\Psi R(z) e^{U/z}$  is a fundamental solution to the $\bT$-equivariant
big QDE \eqref{eqn:qde}.
\item $R$ satisfies the unitary condition $R^T(-z)R(z)=\one$.
\item
\begin{equation}\label{eqn:R-at-zero}
\begin{aligned}
& \lim_{\tQ,\tau''\to 0} R_{\rho,\delta}^{\spa\si,\gamma}(z)\\
=& \frac{\delta_{\rho,\si}}{|G_\si|}\sum_{h\in G_\si}\chi_\rho(h) \chi_\gamma(h^{-1})
\prod_{i=1}^3 \exp\Big( \sum_{m=1}^\infty \frac{(-1)^m}{m(m+1)}B_{m+1}(c^\si_i(h))
(\frac{z}{\w_i(\si)})^m \Big)
\end{aligned}
\end{equation}
\end{enumerate}
\end{theorem}

Each matrix in (2) of Theorem \ref{R-matrix} represents an operator with respect to the classical canonical basis  $\{ \hat{\phi}_{\bsi}: \bsi\in I_\Si\}$.
So $R^T$ is the adjoint of $R$ with respect to the $\bT$-equivariant
Poincar\'{e} pairing $(\ , \ )_{\cX,\bT}$.
The matrix $(\tS_{\bsi'}^{\spa \widehat{\bsi}})(z)$ is of the form
\begin{equation}
  \label{eqn:a-model-S}
\tS_{\bsi'}^{\spa\widehat{\bsi}}(z)
= \sum_{\brho\in I_\Si} \Psi_{\bsi'}^{\spa \brho}
R_{\brho}^{\spa \bsi}(z) e^{u^{\bsi}/z}
=(\Psi R(z))_{\bsi'}^{\spa \bsi} e^{u^{\bsi}/z}
\end{equation}
where $R(z)= (R_{\brho}^{\spa\bsi}(z)) = \one + \sum_{k=1}^\infty R_k z^k$.

We call the unique $R(z)$ in Theorem \ref{R-matrix} the {\em A-model $R$-matrix}.
The A-model $R$-matrix plays a central role in the quantization formula of the descendent potential of $\bT$-equivariant Gromov-Witten
theory of $\cX$.

Let $\bar \psi_i$ be the pullback of the $i$-ith $\psi$-class on $\Mbar_{g,k}$ to $\Mbar_{g,k}(\cX;\beta)$. We define the ancestor potential to be
\[
\cA_\cX(\btau)=\sum_{g\geq 0}\sum_{k\ge 0}\frac{\hbar^{g-1}}{k!}\llangle\bt(\bar \psi_1),\dots,\bt(\bar \psi_k) \rrangle^\cX_{g,k}
\]

Before we move on to the quantization process, let us consider the potential functions of the trivial cohomological field theory $I$. Define the correlator $\langle \rangle_{g,k}^{I}$ to be
$$\langle \tau_{a_1}(\hat{\phi}_{\bsi_1}),\cdots,\tau_{a_k}(\hat{\phi}_{\bsi_k})\rangle_{g,k}^{I}=
\left\{\begin{array}{ll}\Delta_\bsi^{g-1+k/2}(\btau)\int_{\Mbar_{g,k}}\psi_1^{a_1}\cdots\psi_k^{a_k}, &\textrm{if}\quad \bsi_1=\bsi_2=\cdots=\bsi_k=\bsi,\\
0, &\textrm{otherwise}\end{array} \right.$$
where $a_1,\cdots,a_k$ are nonnegative integers. Let
$$\cD_{I}=\exp\big(\sum_{g\geq 0}\sum_{k\geq 0}\sum_{a_1,\cdots,a_k\geq 0}\sum_{\bsi_1,\cdots,\bsi_k}
\frac{\hbar^{g-1}t_{a_1}^{\bsi_1}\cdots t_{a_k}^{\bsi_k}}{a_1!\cdots a_k!}
\langle \tau_{a_1}(\hat{\phi}_{\bsi_1}),\cdots,\tau_{a_k}(\hat{\phi}_{\bsi_k})\rangle_{g,k}^{I}\big).$$

In \cite{Zo}, the third author generalizes Givental's formula for the total descendant
potential of equivariant Gromov-Witten theory of GKM manifolds to GKM orbifolds. When we apply this formula to the case of a toric Calabi-Yau 3-orbifold, we obtain the following theorem：

\begin{theorem}[{Zong \cite{Zo}}]\label{ancestor}
Let $\cA_\cX(\btau)$ be the ancestor potential of $\cX$. Then
$$\cA_\cX(\btau)=\widehat{\Psi}\widehat{R}\cD_{I}.$$
Here $\widehat{\Psi}$ is the operator $\cG(\Psi^{-1}\textbf{q})\mapsto \cG(\textbf{q})$ for any element $\cG$ in the Fock space.
\end{theorem}

Similarly, there is a Givental formula for the descendent potential of $\cX$:

\begin{theorem}\label{descendent-ancestor}
For $2g-2+n> 0$, we have the following relation
$$\llangle \bt(\psi_1),\dots,\bt(\psi_n) \rrangle_{g,n}^{\cX}=\llangle [\cS\bt]_+(\bar\psi_1),\dots,[\cS\bt]_+(\bar\psi_n)\rrangle_{g,n}^{\cX}.$$
Here we consider $\bt=\bt(z)$ as element in $\bH$ and $[\cS\bt]_+(z)$ is the part of $\cS\bt$ containing nonnegative powers of $z$.

\end{theorem}

\subsection{The graph sum formula}
\label{sec:Agraph}
In order to state the graph sum formula, we need to introduce some definitions.

\begin{itemize}
\item We define
$$
S^{\widehat{\underline{\bsi}}}_{\spa \widehat{\underline{\bsi'}} }(z)
:= (\hat{\phi}_{\bsi}(t), \cS(\hat{\phi}_{\bsi'}(t))).
$$
Then $(S^{ \widehat{\underline{\bsi}}  }_{\spa \widehat{\underline{\bsi'}} }(z))$ is the matrix of the $\cS$-operator with
respect to the normalized canonical basis  $\{ \hat{\phi}_{\bsi}(t): \bsi\in I_\Si\} $:
\begin{equation}
\cS(\hat{\phi}_{\bsi'}(t))=\sum_{\bsi\in I_\Si} \hat{\phi}_{\bsi}(t)
S^{\widehat{\underline{\bsi}} }_{\spa \widehat{\underline{\bsi'}} }(z).
\end{equation}
\item We define
$$
S^{\widehat{\underline{\bsi}}}_{\spa \bsi'}(z)
:= (\hat{\phi}_{\bsi}(t), \cS(\phi_{\bsi'})).
$$
Then $(S^{ \widehat{\underline{\bsi}}  }_{\spa \bsi'}(z))$ is the matrix of the $\cS$-operator with
respect to the  basis $\{\phi_{\bsi}:\bsi\in I_\Si\}$ and
$\{ \hat{\phi}_{\bsi}(t): \bsi\in I_\Si\} $:
\begin{equation}
\cS(\phi_{\bsi'})=\sum_{\bsi\in I_\Si} \hat{\phi}_{\bsi}(t)
S^{\widehat{\underline{\bsi}} }_{\spa \bsi'}(z).
\end{equation}
\end{itemize}

Given a connected graph $\Ga$, we introduce the following notation.
\begin{enumerate}
\item $V(\Ga)$ is the set of vertices in $\Ga$.
\item $E(\Ga)$ is the set of edges in $\Ga$.
\item $H(\Ga)$ is the set of half edges in $\Ga$.
\item $L^o(\Ga)$ is the set of ordinary leaves in $\Ga$. The ordinary
leaves are ordered: $L^o(\Ga)=\{l_1,\ldots,l_n\}$ where
$n$ is the number of ordinary leaves.
\item $L^1(\Ga)$ is the set of dilaton leaves in $\Ga$. The dilaton leaves are unordered.
\end{enumerate}

With the above notation, we introduce the following labels:
\begin{enumerate}
\item (genus) $g: V(\Ga)\to \bZ_{\geq 0}$.
\item (marking) $\bsi: V(\Ga) \to I_\Si$. This induces
$\bsi :L(\Ga)=L^o(\Ga)\cup L^1(\Ga)\to I_\Si$, as follows:
if $l\in L(\Ga)$ is a leaf attached to a vertex $v\in V(\Ga)$,
define $\bsi(l)=\bsi(v)$.
\item (height) $k: H(\Ga)\to \bZ_{\geq 0}$.
\end{enumerate}

Given an edge $e$, let $h_1(e),h_2(e)$ be the two half edges associated to $e$. The order of the two half edges does not affect the graph sum formula in this paper. Given a vertex $v\in V(\Ga)$, let $H(v)$ denote the set of half edges
emanating from $v$. The valency of the vertex $v$ is equal to the cardinality of the set $H(v)$: $\val(v)=|H(v)|$.
A labeled graph $\vGa=(\Ga,g,\bsi,k)$ is {\em stable} if
$$
2g(v)-2 + \val(v) >0
$$
for all $v\in V(\Ga)$.

Let $\bGa(\cX)$ denote the set of all stable labeled graphs
$\vGa=(\Gamma,g,\bsi,k)$. The genus of a stable labeled graph
$\vGa$ is defined to be
$$
g(\vGa):= \sum_{v\in V(\Ga)}g(v)  + |E(\Ga)|- |V(\Ga)|  +1
=\sum_{v\in V(\Ga)} (g(v)-1) + (\sum_{e\in E(\Gamma)} 1) +1.
$$
Define
$$
\bGa_{g,n}(\cX)=\{ \vGa=(\Gamma,g,\bsi,k)\in \bGa(\cX): g(\vGa)=g, |L^o(\Ga)|=n\}.
$$

\begin{figure}[h]
\begin{center}
\psfrag{g1}{\small $g=1$}
\psfrag{h1}{\small $k=1$}
\psfrag{g2}{\small $g=1$}
\psfrag{h2-1}{\small $k=0$}
\psfrag{h2-2}{\small $k=2$}
\psfrag{g3}{\small $g=0$}
\psfrag{h3-1}{\small $\ \ k=0$}
\psfrag{h3-2}{\small $k=0$}
\psfrag{h3-3}{\small $k=0$}
\includegraphics[scale=0.6]{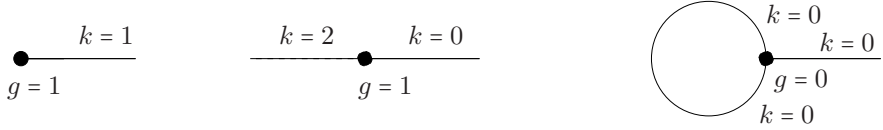}
\caption{A set $\bGa_{1,1}(\cX)$. Each graph here corresponds to several graphs in $\bGa_{1,1}(\cX)$: one can label each vertex by any element in $I_\Si$. Dashed lines are dilaton leafs, whose lowest possible height is $2$.}
\end{center}
\end{figure}

We assign weights to leaves, edges, and vertices of a labeled graph $\vGa\in \bGa(\cX)$ as follows.
\begin{enumerate}
\item {\em Ordinary leaves.}
To each ordinary leaf $l_j \in L^o(\Ga)$ with  $\bsi(l_j)= \bsi\in I_\Si$
and  $k(l)= k\in \bZ_{\geq 0}$, we assign the following descendant  weight:
\begin{equation}\label{eqn:u-leaf}
(\cL^{\bu})^{\bsi}_k(l_j) = [z^k] (\sum_{\bsi',\brho\in I_\Si}
\left(\frac{\bu_j^{\bsi'}(z)}{\sqrt{\Delta^{\bsi'}(t)} }
S^{\widehat{\underline{\brho}} }_{\spa
  \widehat{\underline{\bsi'}}}(z)\right)_+ R(-z)_{\brho}^{\spa \bsi} ),
\end{equation}
where $(\cdot)_+$ means taking the nonnegative powers of $z$.

\item {\em Dilaton leaves.} To each dilaton leaf $l \in L^1(\Ga)$ with $\bsi(l)=\bsi
\in I_\Si$
and $2\leq k(l)=k \in \bZ_{\geq 0}$, we assign
$$
(\cL^1)^{\bsi}_k := [z^{k-1}](-\sum_{\bsi'\in I_\Si}
\frac{1}{\sqrt{\Delta^{\bsi'}(t)}}
R_{\bsi'}^{\spa \bsi}(-z)).
$$

\item {\em Edges.} To an edge connected a vertex marked by $\bsi\in I_\Si$ to a vertex
marked by $\bsi'\in I_\Si$ and with heights $k$ and $l$ at the corresponding half-edges, we assign
$$
\cE^{\bsi,\bsi'}_{k,l} := [z^k w^l]
\Bigl(\frac{1}{z+w} (\delta_{\bsi\bsi'}-\sum_{\brho\in I_\Si}
R_{\brho}^{\spa \bsi}(-z) R_{\brho}^{\spa \bsi'}(-w)\Bigr).
$$
\item {\em Vertices.} To a vertex $v$ with genus $g(v)=g\in \bZ_{\geq 0}$ and with
marking $\bsi(v)=\bsi$, with $n$ ordinary
leaves and half-edges attached to it with heights $k_1, ..., k_n \in \bZ_{\geq 0}$ and $m$ more
dilaton leaves with heights $k_{n+1}, \ldots, k_{n+m}\in \bZ_{\geq 0}$, we assign
$$
 \Big(\sqrt{\Delta^{\bsi}(t)}\Big)^{2g(v)-2+\val(v)}\int_{\Mbar_{g,n+m}}\psi_1^{k_1} \cdots \psi_{n+m}^{k_{n+m}}.
$$
\end{enumerate}

We define the weight of a labeled graph $\vGa\in \bGa(\cX)$ to be
\begin{eqnarray*}
w_A^{\bu}(\vGa) &=& \prod_{v\in V(\Ga)} (\sqrt{\Delta^{\bsi(v)}(t)})^{2g(v)-2+\val(v)} \langle \prod_{h\in H(v)} \tau_{k(h)}\rangle_{g(v)}
\prod_{e\in E(\Ga)} \cE^{\bsi(v_1(e)),\bsi(v_2(e))}_{k(h_1(e)),k(h_2(e))}\\
&& \cdot \prod_{l\in L^1(\Ga)}(\cL^1)^{\bsi(l)}_{k(l)}\prod_{j=1}^n(\cL^{\bu})^{\bsi(l_j)}_{k(l_j)}(l_j).
\end{eqnarray*}

With the above definition of the weight of a labeled graph, we have
the following theorem which expresses the $\bT-$equivariant descendent
Gromov-Witten potential of $\cX$ in terms of graph sum.

\begin{theorem}[{Zong \cite{Zo}}]
\label{thm:Zong}
 Suppose that $2g-2+n>0$. Then
$$
\llangle \bu_1,\ldots, \bu_n\rrangle_{g,n}^{\cX,\bT}=\sum_{\vGa\in \bGa_{g,n}(\cX)}\frac{w_A^{\bu}(\vGa)}{|\Aut(\vGa)|}.
$$
\end{theorem}

\begin{remark}\label{descendent Q}
In the above graph sum formula, we know that the restriction $S^{\widehat{\underline{\brho}} }_{\spa
  \widehat{\underline{\bsi'}}}(z)|_{Q=1}$ is well-defined by Remark \ref{Novikov}. Meanwhile by (1) in Theorem \ref{R-matrix}, we know that the restriction $R(z)|_{Q=1}$ is also well-defined. Therefore by Theorem \ref{thm:Zong}, we have $\llangle \bu_1,\ldots, \bu_n\rrangle_{g,n}^{\cX,\bT}|_{Q=1}$ is well-defined.
\end{remark}

\subsection{Open-closed Gromov-Witten invariants}\label{sec:open-closed-GW}

The BKMP Remodeling Conjecture predicts open-closed Gromov-Witten potentials \eqref{eqn:open-potential} from the B-model.

Let $\cL\subset \cX$ be an outer Aganagic-Vafa Lagrangian brane.
Let $G_0 := G_{\si_0}$ be the stabilizer of the stacky point $\fl_{\si_0}$.
Our notation is similar to that in \cite[Section 5]{FLZ2}. In
particular, the interger $f$ is a framing, and $\bT'_f:=\Ker
(\su_2-f\su_1)$. The morphism
$H^*(B\bT';\bQ)=\bQ[\su_1,\su_2]\to H^*(B\bT'_f; \bQ)=\bQ[\sv]$
is given by $\su_1\mapsto \sv$, $\su_2\mapsto f\sv$.
The weights of $\bT'$-action on $T_{\fl_{\si_0}}\cX$ are
$$
\sw'_1 = \frac{1}{\fr}\su_1 ,\quad \sw'_2 = \frac{\fs}{\fr\fm}\su_1 +\frac{1}{\fm}\su_2, \quad
\sw'_3 = -\frac{\fs+\fm}{\fr\fm}\su_1 -\frac{1}{\fm}\su_2,
$$
so the weights of $\bT'_f$-action on $T_{\fp_{\si_0}}\cX$ are $w_1\sv$, $w_2\sv$, $w_3\sv$, where
$$
w_1= \frac{1}{\fr},\quad w_2=\frac{\fs +\fr f }{\fr\fm} ,\quad
w_3=\frac{-\fm- \fs- \fr f}{\fr\fm}.
$$
Recall that the correlator $\langle \btau^\ell \rangle^{\cX,(\cL,f)}_{g,d,(\mu_1,k_1),\ldots, (\mu_n,k_n)}$ is the
the equivariant open-closed Gromov-Witten invariant defined in Equation \eqref{eqn:open-GW} (see \cite[Section 3]{FLT} for more details). The open Gromov-Witten potential is defined in \eqref{eqn:open-potential}.

We introduce some notation.
\begin{enumerate}
\item Given $d_0\in \bZ$ and $k\in \bmu_\fm$ let $D'(d_0,k)$ be the disk factor
defined by Equation (13) in \cite{FLZ2}, and define
$$
h(d_0,k): = (e^{2\pi\sqrt{-1}d_0w_1},e^{2\pi\sqrt{-1}(d_0w_2-\frac{k}{\fm})},
e^{2\pi\sqrt{-1}(d_0w_3+\frac{k}{\fm})}) \in G_0\subset \bT=(\bC^*)^3.
$$

\item Given $h\in G_0$, define
\begin{eqnarray*}
\Phi^h_0(\tX)&:=&\frac{1}{|G_0|}
\sum_{\substack{(d_0,k)\in \bZ_{\geq 0}\times \bZ_\fm \\ h(d_0,k)=h} } D'(d_0,k) \tX^{d_0} (-(-1)^{-k/\fm})\one'_\frac{-k}{\fm} \\
&=& \frac{1}{|G_0|}
\sum_{\substack{(d_0,k)\in \bZ_{\geq 0}\times \bZ_\fm \\ h(d_0,k)=h} }
\frac{1}{\fm\sv}e^{\sqrt{-1}\pi(d_0 w_3-c_3(h))}
\big(\frac{\sv}{d_0}\big)^{\age(h)-1} \\
&&  \cdot \frac{\Gamma(d_0 (w_1+w_2)+c_3(h)) }{
\Gamma(d_0 w_1-c_1(h)+1)\Gamma(d_0 w_2 -c_2(h)+1)}\tX^{d_0} \one'_\frac{-k}{\fm}.
\end{eqnarray*}
Then $\Phi^h_0(\tX)$ takes values in $\bigoplus_{k=0}^{\fm-1}\bC \sv^{\age(v)-2} \one'_\frac{k}{\fm}$.

For $a\in \bZ$ and $h\in G_0$,  we define
$$
\Phi^h_a(\tX) := \frac{1}{|G_0|}\sum_{\substack{d_0> 0 \\
    h(d_0,k)=h} } D'(d_0,k)(\frac{d_0}{\sv})^a \tX^{d_0}(-(-1)^{-k/\fm})
\one'_\frac{-k}{\fm}.
$$
Then $\Phi^h_a(\tX)$ takes values in $\bigoplus_{k=0}^{\fm-1}\bC \sv^{\age(v)-2-a} \one'_\frac{k}{\fm}$, and
$$
\Phi^h_{a+1}(\tX)=(\frac{1}{\sv}\tX\frac{d}{d\tX})\Phi^h_a(\tX).
$$

\item For $a\in \bZ$ and $\alpha \in G_0^*$, we define
\begin{align}
\label{eqn:txi}
\txi^\alpha_a(\tX):= |G_0|\sum_{h\in G}\chi_\alpha(h^{-1})
\bigl(\prod_{i=1}^3 (w_i\sv)^{1-c_i(h)}\bigr) \Phi^h_a(\tX).
\end{align}
Then $\txi^\alpha_a(\tX)$ takes values in $\bigoplus_{k=0}^{\fm-1}\bC
\sv^{1-a} \one'_\frac{k}{\fm}$. Denote $\txi^\alpha(z,\tX)=\sum_{a\in \bZ_{\geq
  -2}} z^a \txi^\alpha_a(\tX).$

\item Given $h\in G_0=G_{\si_0}$, recall that  $\one_{\si_0,h}$ be characterized by
$\one_{\si_0,h}|_{\fl_{\si_0}} =\delta_{\si,\si_0} \one_h$. We define
$\one_{\si_0,h}^*=|G|\be_h\one_{\si_0, h^{-1}}$, where $\be_h = \prod_{i=1}^3 (w_i\sv)^{\delta_{0,c_i(h)}}$.
\end{enumerate}

With the above notation, the following proposition comes from localization computation (see \cite{FLT}).
\begin{proposition}\label{prop:open-descendant}
\begin{enumerate}
  \item{\emph{Disk invariants}}
  \begin{eqnarray*}
  && \tF_{0,1}^{\cX,(\cL,f)}(\btau,Q;\tX)\\
  &=&\Phi^1_{-2}(\tX) +\sum_{a=1}^p \tau_a \Phi^{h_a}_{-1}(\tX)+ \sum_{a\in \bZ_{\geq 0}}
  \sum_{h\in G_0} \big( \llangle \one_{\si_0,h}^*{\psi}^a \rrangle^{\cX,\bT_f}_{0,1}\big|_{t=\btau}\big)
  \Phi^h_a(\tX) \\
  &=& \frac{1}{|G_0|^2 w_1 w_2 w_3}\Big( \sum_{\alpha\in G_0^*}\txi^\alpha_{-2}(\tX)
  + \sum_{a=1}^p \tau_a \prod_{i=1}^3 w_i^{c_i(h_a)}\sum_{\alpha \in G_0^*}
  \chi_\alpha(h_a)\txi^\alpha_{-1}(\tX)\Big)\Big|_{\sv=1}\\
  &&  +\sum_{a\in \bZ_{\geq 0}}\sum_{\alpha\in G_0^*}
  \big(\llangle \phi_{\si_0,\alpha}{\psi}^a \rrangle^{\cX,\bT_f}_{0,1}\big|_{t=\btau} \big) \txi^\alpha_a(\tX)\\
  &=&[z^{-2}]\sum_{\alpha\in
      G_0^*}S_z(1,\phi_{\si_0,\alpha})\txi^\alpha(z,\tX).
  \end{eqnarray*}

\item{\em (annulus invariants)}
\begin{eqnarray*}
&& \tF_{0,2}^{\cX,(\cL,f)}(\btau; \tX_1,\tX_2)-F^{\cX,(\cL,f)}_{0,2}(0;\tX_1,\tX_2)\\
&=& \sum_{a_1,a_2\in \bZ_{\geq 0}}
\sum_{h_1,h_2\in G_0}
\big(  \llangle \one_{\si_0, h_1}^*{\psi}^{a_1}, \one_{\si_0, h_2}^* {\psi}^{a_1}\rrangle^{\cX,\bT_f}_{0,2}\big|_{t=\btau}\big)
\Phi^{h_1}_{a_1}(\tX_1)\Phi^{h_2}_{a_2}(\tX_2) \\
&=& \sum_{a_1, a_2\in \bZ_{\geq 0}} \sum_{\alpha_1, \alpha_2\in G_0^*}
\big( \llangle {\phi}_{\si_0, \alpha_1}{\psi}^{a_1},
\phi_{\si_0,\alpha_2} {\psi}^{a_1} \rrangle^{\cX,\bT_f}_{0,2}\big|_{t=\btau}\big)
\txi^{\alpha_1}_{a_1}(\tX_1)\txi^{\alpha_2}_{a_2}(\tX_2)\\
\end{eqnarray*}
where
\begin{equation}\label{eqn:annulus-zero}
\begin{aligned}
&(\tX_1\frac{\partial}{\partial \tX_1}+ \tX_2\frac{\partial}{\partial \tX_2}) \tF_{0,2}^{\cX,(\cL,f)}(0;\tX_1,\tX_2)\\
=&|G_0| \Big(\sum_{h\in G_0} \be_h \Phi^h_0(\tX_1) \Phi^{h^{-1}}_0(\tX_2)\Big)\Big|_{\sv=1}\\
=&\frac{1}{|G_0|^2 w_1w_2w_3} \Big(\sum_{\gamma\in G_0^*}\big(\txi^\gamma_0(\tX_1)\txi^\gamma_0(\tX_2)\Big)\Big|_{\sv=1}.
\end{aligned}
\end{equation}

So we have
\begin{align*}
&\tF_{0,2}^{\cX,(\cL,f)}(\btau; \tX_1, \tX_2)\\
=&[z_1^{-1}z_2^{-1}]\sum_{\alpha_1,\alpha_2\in
   G_0^*}V_{z_1,z_2}(\phi_{\si_0, \alpha_1}, \phi_{\si_0,\alpha_2})\txi^{\alpha_1}(z_1,\tX_1)\txi^{\alpha_2}(z_2,\tX_2).
\end{align*}

\item For $2g-2+n>0$,
\begin{eqnarray*}
&& \tF_{g,n}^{\cX,(\cL,f)}(\btau; \tX_1,\ldots, \tX_n)\\
&=&\sum_{a_1,\ldots, a_n\in \bZ_{\geq 0}}
\sum_{h_1,\ldots, h_n\in G_0}
\Big( \llangle \one_{h_1}^*{\psi}^{a_1},\ldots, \one_{h_n}^*{\psi}^{a_n} \rrangle^{\cX,\bT_f}_{g,n}\Big|_{t=\btau}\Big)
\prod_{j=1}^n \Phi^{h_j}_{a_j}(\tX_j) \\
&=&\sum_{a_1,\ldots, a_n\in \bZ_{\geq 0}}
\sum_{\alpha_1,\ldots, \alpha_n\in G_0^*}
\Big(\llangle {\phi}_{\si_0,\alpha_1}{\psi}^{a_1},\ldots
{\phi}_{\si_0,\alpha_n} {\psi}^{a_n}\rrangle^{\cX,\bT_f}_{g,n}\Big|_{t=\btau}\Big)
\prod_{j=1}^n \txi^{\alpha_j}_{a_j}(\tX_j).\\
&=&[z_1^{-1}\dots z_n^{-1}]\sum_{\alpha_1,\dots,\alpha_n\in G_0^*}
\Big( \llangle\frac{\phi_{\si_0,\alpha_1}}{z_1-{\psi}_1},
    \frac{\phi_{\si_0,\alpha_2}}{z_2-{\psi}_2}, \dots, \frac{\phi_{\si_0,\alpha_n}}{z_n-{\psi}_n}\rrangle_{g,n}^{\cX,\bT_f}\Big|_{t=\btau}\Big)
    \prod_{j=1}^n\txi^{\alpha_j}(z_j,\tX_j).
\end{eqnarray*}
\end{enumerate}
\end{proposition}

\begin{remark}
$\tF_{0,2}^{\cX,(\cL,f)}(0;\tX_1,\tX_2)$ is an $H^*(\cB\bmu_m;\bC)^{\otimes 2}$-valued
power series in $\tX_1,\tX_2$ which vanishes at $(\tX_1,\tX_2)=(0,0)$,
so it is determined by \eqref{eqn:annulus-zero}.
\end{remark}

We now combine Section \ref{sec:Agraph} and the above Proposition \ref{prop:open-descendant} to obtain
a graph sum formula for $\tF_{g,n}^{\cX,(\cL,f)}$.  We use the notation in Section \ref{sec:Agraph},
and introduce the notation
\[
\txi^\bsi(z,\tX)=
\begin{cases}
\txi^\alpha(z,\tX), &\text{ if
    $\bsi=(\si_0,\alpha)$},\\
0,& \text{ if $\bsi=(\si,\alpha)$ and $\si\neq \si_0$}.
\end{cases}
\]

\begin{itemize}
\item Given a labelled graph $\vGa\in \bGa_{g,n}(\cX)$, to each ordinary leaf $l_j\in L^o(\gamma)$
with $\bsi(l_j)=\bsi \in I_\Si$ and $k(l_j)\in \bZ_{\geq 0}$ we assign the following weight (open leaf)
\begin{equation}\label{eqn:O-leaf}
(\tcL^{\tX_j})^{\bsi}_k(l_j) = [z^k]\big(\sum_{\brho,\bsi \in I_\Si}
\Big( \txi^\bsi(z,\tX_j) S^{\widehat{\underline{\brho}}}_{\spa  \bsi}\bigg|_{\substack{t=\btau\\ \sw_i=w_i\sv}}
\Big)_+ R(-z)_{\brho}^{\spa \bsi}\bigg|_{\substack{t=\btau\\\sw_i=w_i \sv}} \big).
\end{equation}

\item Given a labelled graph $\bGa_{g,n}(\cX)$, we define a weight
\begin{eqnarray*}
 && \tw_A^{\tX}(\vGa) = \prod_{v\in V(\Ga)} (\sqrt{\Delta^{\bsi(v)}(t)}\bigg|_{\substack{t=\btau\\\sw_i=w_i\sv} })^{2g(v)-2+\val(v)} \langle \prod_{h\in H(v)} \tau_{k(h)}\rangle_{g(v)} \\
&& \cdot \Big( \prod_{e\in E(\Ga)} \cE^{\bsi(v_1(e)),\bsi(v_2(e))}_{k(h_1(e)),k(h_2(e))}\cdot
\prod_{l\in L^1(\Ga)}(\cL^1)^{\bsi(l)}_{k(l)}\Big)\bigg|_{\substack{t=\btau\\ \sw_i = w_i \sv} } \prod_{j=1}^n(\tcL^{\tX_j})^{\bsi(l_j)}_{k(l_j)}(l_j)\\
\end{eqnarray*}
\end{itemize}

Then we have the following graph sum formula for $F_{g,n}^{\cX,(\cL,f)}$.
\begin{theorem}\label{tF-graph-sum}
$$
\tF_{g,n}^{\cX,(\cL,f)} = \sum_{\vGa\in \bGa_{g,n}(\cX)}\frac{\tw_A^{\tX}(\vGa)}{|\Aut(\vGa)|}.
$$
\end{theorem}
\begin{proof}
This follows from Theorem \ref{thm:Zong} and Proposition \ref{prop:open-descendant}.
\end{proof}

\begin{definition}[Restriction to $Q=1$]
We define
\begin{align*}
F_{g}^{\cX}(\btau)&:=\tF_g^\cX(\btau)\vert_{Q=1}\\
F_{g,n}^{\cX,(\cL,f)}(\btau; \tX_1,\ldots, \tX_n):&=
\tF_{g,n}^{\cX,(\cL,f)}(\btau;\tX_1,\ldots,\tX_n)\vert_{Q=1}.
\end{align*}
\end{definition}

\noindent
By Remark \ref{descendent Q} and Propsition \ref{prop:open-descendant},
$F_{g,n}^{\cX,(\cL,f)}$ is well-defined. Theorem \ref{tF-graph-sum} implies
\begin{corollary}
$$
F_{g,n}^{\cX,(\cL,f)} = \sum_{\vGa\in \bGa_{g,n}(\cX)}\frac{w_A^{\tX}(\vGa)}{|\Aut(\vGa)|}.
$$
where $w_A^{\tX}(\vGa)= \tw_A^{\tX}(\vGa)\big|_{Q=1}$.
\end{corollary}

\section{B-model quantization: the topological recursion}

\label{sec:B-model-quantization}

In this section we review the definition of a general spectral curve and the topological recursion on it \cite{EO07}. The variables $\hx,\hy,\hX,\hY$, when restricted to a mirror curve, are indeed the same variables in the previous sections.
\subsection{Spectral curves}
Let $C$ be a smooth affine algebraic curve in $(\bC^*)^2$. The
coordinates $\hX,\hY$ map $\Si$ to the first and the second component of $(\bC^*)^2$ respectively. They
are holomorphic functions on $\Si$ and we require them to be Morse. Let $\hX=e^{-x},\hY=e^{-y}$, where $\hx,\hy$ are well-defined on the universal cover $\bC^2$ of $(\bC^*)^2$. We denote the covering map $\pi$
\begin{align*}
p:\  & \bC^2 \to (\bC^*)^2,\\
         &  (\hx, \hy) \mapsto (e^{-\hx}, e^{-\hy}).
\end{align*}
Let $\tC$ be the lift of $C$ under this map, and let $\oC$ be a choice of smooth compactification of $C$, which is a compact Riemann surface. We denote the genus of $\oC$ by $\fg$.

The intersection pairing $H_1(\oC;\bC)\times H_1(\oC;\bC)\to \bC$ is a symplectic pairing. We choose a \emph{Lagrangian} subspace $\cA$ of $H_1(\oC;\bC)$. A \emph{Torelli marking} on $\oC$ is a choice of symplectic basis $A_1,\dots, A_\fg,B_1,\dots, B_\fg$ in $H_1(\oC;\bC)$, such that $A_i\cap B_j=\delta_{i,j}$
and $A_i\cap A_j =B_i\cap B_j=0$.\footnote{We
    allow such cycles to be non-geometric, i.e. elements in
    $H_1(\oC;\bC)$, and not necessarily in $H_1(\oC;\bZ)$.}
Following the notions from \cite{Fay}, we define the fundamental normalized differential of the second kind (a.k.a. Bergman kernel in Eynard-Orantin \cite{EO07}).
\begin{definition}
The \emph{fundamental normalized differential of the second kind} (abbreviated as \emph{fundamental differential} in this paper)
associated to a Lagrangian subspace $\cA\subset H_1(\oC;\bC)$ is the symmetric meromorphic
form on $(\oC)^2$ satisfying the following conditions.
\begin{itemize}
\item The only pole is the double pole along the diagonal, i.e. given
  any local coordinate $\zeta$ near a point $p\in \oC$, the
  differential $\omega_{0,2}$ has the following form near $(p,p)\in (\oC)^2$
\[
\omega_{0,2}=\frac{d\zeta_1d\zeta_2}{(\zeta_1-\zeta_2)^2} +\text{holomorphic part}.
\]
\item It is normalized by the choice of $A$-cycles in $\cA$
$$
\int_{q\in A} \omega_{0,2}(p,q)=0\text{, $\forall A\in \cA$.}
$$
\end{itemize}
\end{definition}

\begin{definition}
A spectral curve $\cC=(C, \cA)$ consists of the following data:
\begin{itemize}
\item a smooth affine algebraic curve $C$ in $(\bC^*)^2$ where the coordinate functions $\hX,\hY$ are holomorphic Morse;
\item the compactification $\oC$ of $C$ as a smooth projective curve;
\item a choice of Lagrangian subspace $\cA\subset H_1(\oC;\bC)$;
\item a fundamental normalized differential of the second kind $B$ on $\oC$ with respect to such choice of $\cA$.
\end{itemize}
\end{definition}

Fix a spectral curve $\cC$. Since $\hX$ (and then $\hx$) is Morse, the
critical points ($d\hx=0$) form a finite set $\{p_\bsi: \bsi \in
I_\cC\}$ -- here $I_\cC$ is the index set for the ramification points on $C$. Define the
Liouville form $\Phi=\hy d\hx=-\log Y d\hx$. It is a well-defined holomorphic
form on $\tC$.

At each critical point $p_\bsi$, we define the local
coordinate $\zeta_\bsi$ by
\[
\hx=\zeta_\bsi^2+\hx_{0,\bsi},
\]
where $\hx_{0,\bsi}$ is the value of $\hx$ at $p_\bsi$ (well-defined up to an integral multiple of $2\pi\sqrt{-1}$). For any $p$ near $p_\bsi$, let $\bar p$ be the point on $C$ such that
$\zeta_\bsi(\bar p)=-\zeta_\bsi(p)$. We also denote $\hy(p_\bsi)=\hy_{0,\bsi}$, and
\[
\hy=\hy_{0,\bsi}+\sum_{i=d}^\infty h^\bsi_d \zeta_\bsi^d.
\]
By the smoothness of the curve $C$, $h^\bsi_1 \neq 0$ for all $\bsi$. Here $\hat y_{0,\bsi}$ is also well-defined up to an integral multiple of $2\pi\sqrt{-1}$.

\subsection{Eynard-Orantin's topological recursion}

\begin{definition}
\label{def:EO-recursion}
Given a spectral curve $\cC$, the Eynard-Orantin recursive algorithm defines a sequence of symmetric
meromorphic forms $\omega_{g,n}$ on $(\oC)^n$ for $g\in \bZ_{\geq 0},
n\in \bZ_{>0}$ as follows.
\begin{itemize}
\item Initial conditions:
\[
\omega_{0,1}=\Phi=\hat y d \hat x,\quad \text{$\omega_{0,2}$ given by the choice of $\cA$}.
\]
\item Recursive algorithm:
\begin{align*}
\omega_{g,n}(p_1,\dots,p_n)=\sum_{p' \in I_\cC} \Res_{p=p'}
\frac{\int_{\xi=p}^{\bar p}B(p_n, \xi)}{2(\Phi(p)-\Phi(\bar p))} \big
  ( \omega_{g-1,n+1}(p,\bar p,p_1,\dots,p_{n-1})
  \\+\sum^\prime_{g_1+g_2=g,\ I\sqcup J=\{1,\dots,n-1\}}\omega_{g_1,|I|+1}(p,
p_I)\omega_{g_2,|J|+1}(\bar p, p_J)\big ).
\end{align*}
Here the sum symbol $\overset{\prime}{\sum}$ excludes the case $(g_1,|I|)=(0,1),(0,n-1),(g,1)$ or $(g,n-1)$.
\end{itemize}
\end{definition}
The resulting $\omega_{g,n}$ are well-behaved.
\begin{proposition}[Appendix A of \cite{EO07}]
When $2g-2+n>0$, the poles of $\omega_{g,n}$ are at $d\hx_i=0$ (critical points), where
$\hx_i=-\log \hX_i$ is the $\hx$-coordinate on $i$-th copy of $C^n$.
\end{proposition}

\begin{remark}
  The original recursion algorithm in \cite{EO07} sets $\omega_{0,1}=0$ while it does not exclude these four special cases by a special sum symbol $\sum'$ -- it is the same as the recursion here. Another different convention is to introduce a minus sign in the recursion kernel, i.e. using $2(\Phi(\bar p)-\Phi(p))$ in the denominator. Adopting this convention is equivalent to changing all $\omega_{g,n}$ to $(-1)^{g-1} \omega_{g,n}$. We stick to the convention in Definition \ref{def:EO-recursion} throughout this paper.
\end{remark}

\subsection{Differential forms on the spectral curve}

Given a spectral curve $\cC=(C,\cA)$, for each ramification point $p_\bsi,\ \si\in I_\cC$, we associate a path $\gamma_\bsi$ as the Lefschetz thimble
\[
\hx(\gamma_\bsi)=[\hx_{0,\bsi},+\infty)
\]

Following \cite{E11, EO15}, given any $\bsi \in I_\cC$ and $d\in \bZ_{\geq 0}$, define
$$
\theta_{\bsi}^d(p):= (2d-1)!! 2^{-d}\Res_{p'\to p_{\bsi}}
\omega_{0,2}(p,p')\zeta_\bsi^{-2d-1}.
$$
Then $\theta_{\bsi}^d$ satisfies the following properties.
\begin{enumerate}
\item $\theta_{\bsi}^d$ is a meromorphic 1-form on $\oC$ with
a single pole of order $2d+2$ at $p_{\bsi}$.

\item In local coordinate $\zeta_{\bsi} =\sqrt{ \hat{x}-\hat x_{0,\si}}$ near $p_{\bsi}$,
$$
\theta_{\bsi}^d = \Big( \frac{-(2d+1)!!}{2^d \zeta_\bsi^{2d+2}}
+ f(\zeta_\bsi)\Big) d\zeta_\bsi,
$$
where $f(\zeta_\bsi)$ is analytic around $p_{\bsi}$.
The residue of $\theta^0_{\bsi}$ at $p_{\bsi}$ is zero,
so $\theta^d_\bsi$ is a differential of the second kind.
\item
\[
\int_{A}\theta_\bsi^d=0,\ \forall A\in \cA
\]
\end{enumerate}
The meromorphic 1-form $\theta_{\bsi}^d$ is characterized by the above
properties; $\theta_{\bsi}^d$ can be viewed as a section in
$H^0(\oC, \omega_{\oC}((2d+2) p_{\bsi})
)$. 

\begin{remark}\label{sign}
The meromorphic 1-form $\theta_{\bsi}^d$ corresponds to $d\xi_{\bsi}^d$ in \cite{E11} and \cite{E14}. However there is a sign error in \cite[equation (4.7)]{E14}: there is no minus sign on the right hand side. The expansion in \cite[equation (4.8)]{E14} is consistent with (2) in the above properties of $\theta_{\bsi}^d$. The expansion in \cite[equation (4.8)]{E14} will be consistent with \cite[equation (4.7)]{E14} after correcting this sign error.

Besides, we use the notation $\theta_{\bsi}^d$ instead of $d\xi_{\bsi}^d$ since in general $\theta_{\bsi}^d$ is not an exact form.

\end{remark}

\subsection{Graph sum formula from Dunin-Barkowski--Orantin--Shardin--Spitz \cite{DOSS}}

Following \cite{DOSS}, the B-model invariants $\omega_{g,n}$ are expressed in terms of graph sums. We first introduce some notation.
\begin{itemize}
\item  For any $\bsi \in I_{\Sigma}$, we define
\begin{equation}
\check{h}^{\bsi}_{k} :=\frac{(2k-1)!!}{2^{k-1}}h^\bsi_{2k-1}.
\end{equation}
Then by expanding $\hat x$ and compute the integral term-by-term (for $u>0$)
$$
\check{h}^{\bsi}_k = [u^{1-k}]\frac{u^{3/2}}{\sqrt{\pi}} e^{u\hat{x}_{0,\bsi}}
\int_{p\in \gamma_{\bsi}}e^{-u \hat{x}(p)}\Phi(p).
$$

\item For any $\bsi,\bsi'\in I_\cC$, we expand
$$
\omega_{0,2}(p_1,p_2) =\Big( \frac{\delta_{\bsi,\bsi'}}{ (\zeta_\bsi-\zeta_{\bsi'})^2}
+ \sum_{k,l\in \bZ_{\geq 0}} B^{\bsi,\bsi'}_{k,l} \zeta_\bsi^k \zeta_{\bsi'}^l \Big) d\zeta_\bsi d\zeta_{\bsi'},
$$
near $p_1=p_{\bsi}$ and $p_2=p_{\bsi'}$, and define
\begin{equation}\label{eqn:BcheckB}
\check{B}^{\bsi,\bsi'}_{k,l} := \frac{(2k-1)!! (2l-1)!!}{2^{k+l+1}} B^{\bsi,\bsi'}_{2k,2l}.
\end{equation}
Then
\begin{eqnarray*}
\check{B}^{\bsi,\bsi'}_{k,l}&=&[u^{-k}v^{-l}]\left(\frac{uv}{u+v}(\delta_{\bsi,\bsi'}
-\sum_{\bgamma\in I_\cC} f^{\ \bsi}_\bgamma(u)f^{\ \bsi'}_\bgamma(v))\right)\\
&=&[z^{k}w^{l}]\left(\frac{1}{z+w}(\delta_{\bsi,\bsi'}
-\sum_{\bgamma\in I_\cC} f^{\ \bsi}_\bgamma(\frac{1}{z})f^{\ \bsi'}_\bgamma(\frac{1}{w}))\right).
\end{eqnarray*}
\end{itemize}

Given a labeled graph $\vGa \in \bGa_{g,n}(I_\cC)$ with
$L^o(\Ga)=\{l_1,\ldots,l_n\}$, and $\vs=(s_1,\dots,s_n)\in \oC^n$, 
we define its weight to be
\begin{eqnarray*}
w_B^{\vs}(\vGa)&=& (-1)^{g(\vGa)-1}\prod_{v\in V(\Gamma)} \Big(\frac{h^{\balpha(v)}_1}{\sqrt{-2}}\Big)^{2-2g(v)-\val(v)} \langle \prod_{h\in H(v)} \tau_{k(h)}\rangle_{g(v)}
\prod_{e\in E(\Gamma)} \check{B}^{\balpha(v_1(e)),\balpha(v_2(e))}_{k(e),l(e)} \\
&& \cdot \prod_{l\in \cL^1(\Gamma)}(\check{\cL}^1)^{\balpha(l)}_{k(l)}
\prod_{j=1}^n (\check{\cL}^\vs_B)^{\balpha(l_j)}_{k(l_j)}(l_j)
\end{eqnarray*}
where
\begin{itemize}
\item (dilaton leaf)
$$
(\check{\cL}^1)^{\bsi}_k = -\frac{1}{\sqrt{-2}}\check{h}^{\bsi}_k.
$$
\item (descendant leaf)
$$
(\check{\cL}^\vs_B)^{\bsi}_k(l_j) =  \frac{-1}{\sqrt{-2}} \theta_{\bsi}^k(p_j)
$$
\end{itemize}
Notice that $\check B^{\bsi,\bsi'}_{k,l}$ plays the role of the edge contribution, while the vertex contribution is $(-1)^{g(\vGa)-1}\prod_{v\in V(\Gamma)} \Big(\frac{h^{\balpha(v)}_1}{\sqrt{-2}}\Big)^{2-2g(v)-\val(v)} \langle \prod_{h\in H(v)} \tau_{k(h)}\rangle_{g(v)}$.

In our notation \cite[Theorem 3.7]{DOSS} is equivalent to:
\begin{theorem}[Dunin-Barkowski--Orantin--Shadrin--Spitz \cite{DOSS}] \label{thm:DOSS}
For $2g-2+n>0$,
$$
\omega_{g,n}(s_1,\dots,s_n) = \sum_{\Gamma \in \bGa_{g,n}(I_\cC)}\frac{w_B^{\vs}(\vGa)}{|\Aut(\vGa)|}.
$$
\end{theorem}

\begin{remark}
Our convention for the factors in the above graph sum formula is different from that in \cite[Theorem 3.7]{DOSS}. We summarize the following convention differences.
\begin{enumerate}
\item Our $\check{h}^{\bsi}_{k}$ is $\frac{1}{2^k}$ times the $\check{h}^{\bsi}_{k}$ in \cite[Theorem 3.7]{DOSS}.\\

\item Our $\check{B}^{\bsi,\bsi'}_{k,l}$ is $\frac{1}{2^{k+l+1}}$ times the $\check{B}^{\bsi,\bsi'}_{k,l}$ in \cite[Theorem 3.7]{DOSS}.\\

\item Our $\theta_{\bsi}^k$ is $\frac{1}{2^k}$ times the $d\xi^{\bsi}_k$.

\end{enumerate}
\end{remark}

\subsection{An equivalent graph sum formula from Eynard}
In \cite[Theorem 5.1]{E11}, Eynard obtains a graph sum formula for $\omega_{g,n}$ on a general spectral curve. In this subsection, we show that this graph sum formula is equivalent to the graph sum formula in \cite[Theorem 3.7]{DOSS} by direct computation.

The formula in \cite[Theorem 5.1]{E11} sums over all the stable degeneracies of the moduli space $\cM_{g,n}$. So by the dual graph of a stable curve, this is equivalent to summing over the graphs in $\bGa_{g,n}(I_\cC)$. The ordinary leaf term in our graph sum formula matches the factor $d\xi$ in \cite[Theorem 5.1]{E11} up to the factor $\frac{-1}{\sqrt{-2}}$ (But we should be careful with the sign problem in $d\xi$ see Remark \ref{sign}). The $\hat B$ factor in \cite[Theorem 5.1]{E11} corresponds to our $\check{B}^{\bsi,\bsi'}_{k,l}$ which appears in the edges factor. The $\hat B$ factor which appears directly in the graph sum formula in \cite[Theorem 5.1]{E11} corresponds to the loop factor i.e. an edge connecting one vertex. The correlator in \cite[equation (5.1)]{E11} involves the $\kappa$ classes by \cite[equation (5.4)]{E11}. The second factor in \cite[equation (5.4)]{E11} gives us the factors for an edge which connects two different vertices.

The only nontrivial factor is the first factor in \cite[equation (5.4)]{E11} which involves the $\kappa$ classes and $\tilde{t}_{\bsi,k}$ which is called the dual time in \cite{E11}. By \cite[equation (6.7)]{E11}, we have
\begin{equation}\label{dual time}
e^{-\sum_{k=0}^{\infty}\tilde{t}_{\bsi,k} u^{-k}}=2\sum_{k=0}^{\infty}\check{h}_{k+1}^\bsi u^{-k}.
\end{equation}
In particular, $e^{-\tilde{t}_{\bsi,0}}=2\check{h}_{1}^\bsi$. Notice that $\kappa_0=2g(v)-2+\val(v)$ on $\Mbar_{g(v),\val(v)}$ and so we have

$$
e^{\sum_{k=0}^{\infty}\tilde{t}_{\bsi,k} \kappa_k}=(2\check{h}_{1}^\bsi)^{(2-2g(v)-\val(v))}e^{\sum_{k=1}^{\infty}\tilde{t}_{\bsi,k} \kappa_k}.
$$
The first factor on the right hand side is consistent with our vertex factor up to the power of 2 and roots of unity. For the second factor, we apply (\ref{dual time}) and \cite[Lemma 2.2]{LX09} and it is easy to see that this will give us the dilaton leaf factor and the correlator in our vertex factor. In the end, one only needs to notice the identity
$$
\sum_{v\in\Gamma}(2g(v)-2+\val(v))=2g-2+n
$$
to match the factors involving powers of 2 and the roots of unity in \cite[Theorem 5.1]{E11} and in Theorem \ref{thm:DOSS}.

\subsection{B-model open potentials from the Eynard-Orantin's recursion}

For any $q\in \fB^\circ$, the mirror curve $C_{q}$ comes with a compactification $\oC_{q}\subset \bS_P$. The images of $A_1,\dots, A_\fp\in H_1(C_q;\bC)$ in $H_1(\oC_q;\bC)$ under the map $H_1(C_q;\bC)\to H_1(\oC_q;\bC)$ span a Lagrangian subspace in $H_1(\oC_q;\bC)$. We also have $\hat X$ and $\hat Y$ as two holomorphic Morse functions on $C_q$ and they are meromorphic on $\oC_q$. Thus this is a spectral curve for $q\in \cB^\circ$, denoted by $\cC_q$.

Let $\rho_q^\ell: D_\delta \hookrightarrow D^\ell_q$ be an embedding of a small disk $\{|\underline X|<\delta\}$ into $D^\ell_q$ by mapping $\underline X$ to the point whose $\hat X$-coordinate is $\underline X$, while $\rho_q^{\ell_1,\dots,\ell_n}=\rho_q^{\ell_1}\times\dots\times \rho_q^{\ell_n}: (D_\delta)^n \to D^{\ell_1}_q\times \dots\times D^{\ell_n}_q\subset (\oC_{q})^n$. The Eynard-Orantin's topological recursion (Definition \ref{def:EO-recursion}) produces symmetric meromorphic forms $\omega_{g,n}$ on $(\oC_q)^n$. From them we define the B-model open potentials as below.

\begin{enumerate}
  \item (disk invariants) When $q=0$, $\hat Y(\bar p_\ell)^\fm=-1$ for all $\ell$.  On each $D^\ell_q$, $\hat Y$ is a holomorphic function. Since $D^\ell_q$ is very small, the real part of $\hat Y$ on it is still negative -- so we can choose a branch of logarithm $\log:\bC\setminus [0,\infty)\to \bC$, and define a function $\hat y^\ell$ on $D^\ell_q$
  \[
  \hat y^\ell=-\log \hat Y.
  \]
  The function $\hat y^\ell-\hat y^\ell(\bar p^\ell)$ does not depend on the choice of logarithm. So define
  \[
  \check F_{0,1}(q;\hat X)=-\sum_{\ell\in \bmu^*_\fm} \int_{0}^{\hat X} (\rho^\ell_q)^* (\hat y^\ell-\hat y^\ell(\bar p_\ell)) \frac{dX'}{X'}\psi_\ell.
  \]

  \item (annulus invariants) The meromorphic form $\omega_{0,2}$ is not holomorphic on $D_q\times D_q\subset \oC_q\times \oC_q$. One removes the singular part, and defines the following
  \[
  \check F_{0,2}(q;\hX_1,\hX_2):=\sum_{\ell_1,\ell_2\in \bmu^*_\fm} \int_0^{\hat X_1}\int_{0}^{\hat X_2}\left( (\rho_q^{\ell_1,\ell_2})^*\omega_{0,2}-\frac{dX_1'dX_2'}{(X_1'-X_2')^2}\right)\psi_{\ell_1}\otimes\psi_{\ell_2}.
  \]
  \item (stable cases: $2g-2+n>0$) $(\rho_q{\ell_1,\dots,\ell_n})^*\omega_{g,n}$ is holomorphic on $(D_q)^n$. We define
  \[
  \check F_{g,n}(q;\hX_1,\dots,\hX_n):=\sum_{\ell_1,\dots,\ell_n\in \bmu^*_\fm}\int_{0}^{\hX_1}\dots \int_{0}^{\hX_n}(\rho_q^{\ell_1,\dots,\ell_n})^*\omega_{g,n}\psi_{\ell_1}\otimes\dots\otimes\psi_{\ell_n}.
  \]
\end{enumerate}

\section{Comparing the graph sums: proving the Remodeling Conjecture}
\label{sec:comparing}

In this section we survey the proof from \cite{FLZ3} on how to match the graph sums. The key idea is that \emph{graph sum ingredients are genus $0$ information, and the genus $0$ open-closed mirror theorem can be used to match them}.

Throughout this section we use the $1$-dimensional torus $\bT'_f$ for all equivariant cohomology.

\subsection{The statement of open mirror theorems (BKMP Remodeling Conjecture)}

There is an open mirror map
\begin{equation}
  \label{eqn:open-mirror-map}
\tX=\hX(1+O(q)).
\end{equation}
We do not give its explicit formula here, which can be directly written down since it is a solution to certain GKZ system with a prescribed asymptotic behavior \cite{LM} (see also \cite{FLT} for orbifolds). We call Equation \eqref{eqn:closed-mirror-map}  and \eqref{eqn:open-mirror-map} the \emph{open-closed mirror map}.

\begin{remark}
  This open mirror map is the same as Equation \eqref{eqn:open-mirror-map-integral}, which is the geometric origin of \eqref{eqn:open-mirror-map}. One does not need this fact to prove the related mirror symmetry statements, such as for disk invariants and higher-genus (BKMP).
\end{remark}

The full genus mirror symmetry statements are the following
\begin{theorem}
  Under the open-closed mirror map,
  \begin{itemize}
    \item (Aganagic-Klemm-Vafa's conjecture on disk invariants, the general case proved in \cite{FLT})
    \begin{equation}
      \label{eqn:disk-mirror-theorem}
    \check F_{0,1}(q;\hat X)=|G_0|F_{0,1}^{\cX,(\cL,f)}(\btau,\tX);
    \end{equation}
    \item (BKMP's Remodeling Conjecture)
    \begin{itemize}
      \item (Annulus invariants)\[\check F_{0,2}(q;\hat X_1,\hat X_2)=-|G_0|^2 F_{0,2}^{\cX,(\cL,f)}(\btau;\tX_1,\tX_2);\]
      \item (Stable cases) For $2g-2+n>0$,
      \[\check F_{g,n}(q;\hat X_1,\dots,\hat X_n)=|G_0|^n (-1)^{g-1+n} F_{g,n}^{\cX,(\cL,f)}(\btau;\tX_1,\dots, \tX_n).\]
    \end{itemize}
  \end{itemize}
\end{theorem}

\subsection{Graph sum components: vertices}
\label{sec:vertices}

As discussed in Section \ref{sec:identification-frobenius}, the identification of Frobenius algebras implies that the length of the canonical basis matches, as in Equation \eqref{eqn:length-matching}
\[
\frac{h^\balpha_1}{\sqrt{-2}}=\frac{1}{\sqrt{\Delta^\balpha(\btau)}}.
\]
This equates the A-model vertex contribution
\[
\left(\frac{1}{\sqrt{\Delta^{\balpha(v)}(\btau))}}\right)^{2-2g(v)-\val(v)}\vert_{Q=1}\langle \prod_{h\in H(v)} \tau_{k(h)}\rangle_{g(v)}
\]
with the B-model vertex contribution
\[
\left(\frac{h_1^{\balpha(v)}}{\sqrt{-2}}\right)^{2-2g(v)-\val(v)}\langle \prod_{h\in H(v)} \tau_{k(h)}\rangle_{g(v)}
\]

\subsection{Oscillatory integrals and $\check R$-matrices}

The functions $V_\balpha$ are canonical basis for the Landau-Ginzburg B-model $((\bC^*)^3,W^{\bT'})$.

The roles of the meromorphic $1$-forms $\theta_\balpha^0$ are similar to the canonical basis $V_\balpha$, in the following sense.
\begin{proposition}[Dimensional reduction]
  \label{prop:dr}
  For any $\balpha\in I_\Si$ (i.e. index sets of the canonical basis  are the same for both A and B-models since the Frobenius algebra are isomorphic), we have a $3$-cycle (non-compact) $\Gamma_\balpha\ni P_\balpha$ in $(\bC^*)^3$, and a non-compact $1$-cycle $\gamma_\balpha\ni p_\balpha$ in $C_q$ such that
  \begin{align*}
    \int_{\Gamma_\balpha} e^{-\frac{W^{\bT'}}{z}} \Omega& = 2\pi\sqrt{-1}\int_{\gamma_\balpha}e^{-\frac{\hat x}{z}} \Phi;\\
    \int_{\Gamma_\balpha} e^{-\frac{W^{\bT'}}{z}} \Vbar_{\bbeta} \Omega &= 2\pi\sqrt{-1} z^2 \int_{\gamma_\balpha} e^{-\frac{\hat x}{z}} \frac{h_1^{\bbeta}\theta^0_\bbeta}{2}.
  \end{align*}
  Here $z>0$, and  $\gamma_\balpha=\{\hat x\in [\hat x_{0,\balpha},\infty)\}$.
\end{proposition}
\begin{remark}
  The modified canonical basis $\Vbar_{\balpha}$ is a quadratic polynomial in $z$ -- the degree $0$ term is $V_\balpha$. There is a unique way to construct $\Vbar_\balpha$ from $V_\alpha$. See \cite[Section 6.1]{FLZ3}
\end{remark}
\begin{remark}
It is expected that these oscillatory integrals \cite{Fa16,FLZ2-5}, under the open-closed mirror map $\btau=\btau(q)$, should be equal to
\[
\llangle \phi_\bbeta(\btau),\frac{\kappa(\cE_\balpha)}{z-\psi}\rrangle_{0,2},
\]
where $\kappa(\cE_\balpha)$ is certain characteristic classes invovling Gamma functions of a \emph{mirror} coherent sheaf $\cE_\balpha$ (mirror to $\gamma_\balpha$ or $\Gamma_\balpha$) on $\cX$ .
\end{remark}

There are related definition (following \cite{EO15})
\begin{align}
  \label{eqn:b-model-S}
 &\int_{\Gamma_{\bsi}} e^{-\frac{W^{\bT'}}{z}} \frac{\sqrt{-2}\Vbar_{\bsi'}}{h_1^{\bsi'}}\sim(-2\pi z)^\frac{3}{2}  \check R_{\bsi'}^{\ \bsi}(z) e^{-\frac{\hat x_{0,\bsi}}{z}}
 \\
 &f_{\bsi'}^{\spa \bsi}(u) \sim
 \frac{e^{\hat x_{\bsi,0}}}{2\sqrt{\pi u}}
 \int_{\gamma_{\bsi}} e^{-u\hat{x}}\theta^0_{\bsi'}.
 \nonumber
\end{align}

The dimensional reduction Proposition \ref{prop:dr} implies
\[
\check R_{\balpha}^{\  \bbeta}(z)=f_{\balpha}^{\ \bbeta}(\frac{1}{z}).
\]
The A-side QDE is \eqref{eqn:qde}, while the B-side equation comes from the following simple calculus
\[
-z \frac{\partial}{\partial \tau_i}\int e^{-\frac{W^{\bT'}}{z}} \omega=(\frac{\partial W^{\bT'}}{\partial \tau_i})\int e^{-\frac{W^{\bT'}}{z}}\omega.
\]
Some remarks on this simple fact:
\begin{itemize}
  \item The integral is over any flat half-dimensional cycle on which the integral is converging.
  \item $\omega$ needs to be \emph{flat}, i.e. it does not depend on the parameter $q$ (or $\btau$, differing with $q$ by a mirror map), or invariant under the Gauss-Manin connection. Notice $V_\bsi$ or $\Vbar_\bsi$ is not flat -- they are canonical basis and vary with the parameters $q$. So the integral in Equation \eqref{eqn:b-model-S} does not satisfy this equation \emph{per se}. A linear combination of $V_\bsi$ (with coefficients dependent on $q$ or $\btau$) which produces a flat form does satisfy this equation -- for example, using the canonical to flat change of basis matrices $\Psi_{\bsi'}^{\ \bsi}$ as in Section \ref{sec:A-S}.
  \item The identification of the genus $0$ mirror symmetry identifies $\frac{\partial W^{\bT'}}{\partial \tau_i}$ with $H_i$ -- so the above differential equation is the same as the A-side QDE \eqref{eqn:qde}.
\end{itemize}

The A-model $S$-matrix \eqref{eqn:a-model-S} and the B-model oscillatory integral \eqref{eqn:b-model-S} satisfy the QDE \eqref{eqn:qde}. By a theorem in \cite{Du93,Gi97}, we know $R$ and $\check R$ are uniquely determined up to constants. We can fix these constants at the large radius limit point $q=0$. The value of $\check R$ is explicitly computed in \cite{FLZ2} (here $\balpha=(\alpha,\gamma),\bbeta=(\beta,\delta)$)
\begin{align*}
  &\check R_{\balpha}^{\ \bbeta}(-z)|_{q=0}\\
  =& \frac{\delta_{\alpha,\beta}}{|G_\alpha|}\sum_{h\in G_\alpha}\chi_\gamma(h) \chi_\delta(h^{-1})
  \prod_{i=1}^3 \exp\Big( \sum_{m=1}^\infty \frac{(-1)^m}{m(m+1)}B_{m+1}(c^\alpha_i(h))
  (\frac{z}{\w_i(\alpha)})^m \Big)
\end{align*}
which is precisely $R_{\balpha}^{\ \bbeta}(z)|_{q=0}$ given in Equation (\ref{eqn:R-at-zero}). We have matched $R_{\balpha}^{\ \bbeta}(z)=\check R_\balpha^{\ \bbeta}(-z)$.

\subsection{Open leafs}

The matching of A and B-model $R$-matrices $R_{\balpha}^{\ \bbeta}(z)$ with $\check R_{\balpha}^{\ \bbeta}(-z)$ identifies all graph components other than open leafs. Open leaf requires more than $R$-matrices. Identifying $R$-matrices we rely on that the Frobenius structures on both sides are equal, while identifying open leafs we rely on genus $0$ \emph{open} mirror theorem, namely identifying the disk Gromov-Witten potential with $\Phi=\hat y d\hat x$ ($\Phi$ corresponds to disk invariants with \emph{no inseration}.)

Recall the A-model open leaf at vertex $\bsi$ with height $k$ ($k\geq 0$) is (Equation \eqref{eqn:O-leaf})
\begin{equation}
(\cL^\tX)^\bsi_k(l_j)=
[z^k]\left(\sum_{\bsi',\brho \in I_\bSi}(\txi^{\si'}(z,\tX) S(\hat \phi_\brho(\btau),\phi_{\bsi'}))_+R(-z)^{\ \bsi}_\brho \right).
\label{eqn:open-leaf}
\end{equation}
The open leaf weight, as a power series in $\tX$, determines each other for different height $k$ -- they are related by (c.f. Equation \eqref{eqn:txi} for the definition of $\txi$)
\[
(\cL^\tX)^\bsi_{k+1}=(\frac{1}{\sv}\tX\frac{d}{d\tX})(\cL^\tX)^\bsi_{k}.
\]
The localization computation says
\[
[z^0]\sum_{\bsi'\in I_\bSi} \txi^{\bsi'}(z,\tX)S(1,\phi_{\bsi'})=(\tX\frac{d}{d\tX})^2 F_{0,1}^{\cX,(\cL,f)}(\btau,\tX),
\]
while the open mirror theorem of \cite{FLT} further relates $F_{0,1}^{\cX,(\cL,f)}(\btau, \tX)$ to
\[
|G_0|F_{0,1}^{\cX,(\cL,f)}(\btau,\tX)=\check F_{0,1}(q;\hat X).
\]
So immediately one obtains
\[
[z^0]\sum_{\bsi'\in I_\bSi} \txi^{\bsi'}(z,\tX)S(1,\phi_{\bsi'})=-\frac{1}{|G_0|}\sum_{\ell\in \bmu_\fm^*}(\rho^\ell_q)^* (\frac{d\hat y}{d\hat x})\psi_\ell.
\]
The part in $()_+$ of the open leaf \eqref{eqn:open-leaf} involves the insertion of $\hat\phi_\rho(\btau)$, while the genus zero mirror theorem only deals with the insertion of $1$ in the $S$-function. However, since $S(1,\phi_{\bsi'})$ is a solution to QDE (Equation \eqref{eqn:qde}), taking derivatives with respect to $\btau$ we obtain the following
\[
-z\frac{\partial}{\partial \tau_i}=S(H^{\bT}_i,\phi_{\bsi'}).
\]

Since $\{H^\bT_i\}_{i=1}^\fp$ multiplicatively generate the cohomology, we choose $a_i,b_i$ ($i=1,\dots,\fg$) such that $H_{a_i}\star H_{b_i}$ form a basis of $H^4(\cX)$. This choice is not unique. In principle one can express in terms of
\[
\hat \phi_{\bsi}(\btau)=\sum_{i=1}^\fg \hat A_\bsi^i(\btau) H^\bT_{a_i}\star H_{b_i}^\bT+\sum_{a=1}^\fp \hat B^a_\bsi(\btau) H_a^\bT+\hat C_\bsi(\btau) \one.
\]
The coefficients are complicated. An very important observation is that we have the \emph{same} expression of $\theta_\bsi$ with \emph{same} coefficients
\[
\frac{\theta_\bsi}{\sqrt{-2}}=\sum_{i=1}^\fg \hat A^i_\bsi(\btau(q)) \frac{\partial^2 \Phi}{\partial \tau_{a_i}\partial \tau_{b_i}}+\sum_{a=1}^\fp \hat B^a_{\bsi}(\btau(q)) d(\frac{\frac{\partial \Phi}{\partial \tau_a}}{d\hat x })+\hat C_\bsi(\btau(q))d(\frac{d\hat y}{d\hat x}).
\]
This allows us to compute the $()_+$-part in \eqref{eqn:open-leaf}. We end up with
\[
|G_0|\left(\sum_{\si'\in I_\bSi}\txi^{\bsi'}(z,\tX)S(\hat\phi_\bsi(\btau),\phi_{\bsi'})\right)_+=\sum_{\ell\in \bmu_\fm^*}\int_{0}^{\hat X} \frac{(\rho^\ell_q)^*\hat \theta_\bsi(z)}{\sqrt{-2}}.
\]
We see that $\theta_\bsi^0$ plays the role of disk invariants where one has \emph{a closed insertion $\hat \phi_\bsi(\btau)$}, after some constant factor. Therefore, comparing the leaf terms of A-model and B-model graph sums
\[
|G_0|(\cL_O^\tX)^\bsi_k(l_j)=-(\check \cL_O^{\hX} )^\bsi_k(l_j).
\]

Thus all graph components are matched -- the factor $|G_0|$ here results in the factor $|G_0|^n$ in the conjecture, and the sign contributes to $(-1)^n$.

\begin{figure}[h]
\begin{center}
\psfrag{p}{\tiny \hspace{-0.2cm}$\hat \phi_\bsi(\btau)$}
\includegraphics[scale=0.6]{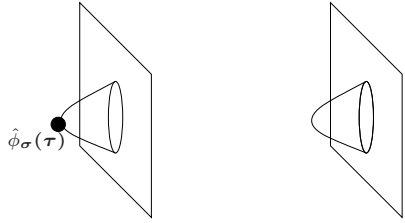}
\caption{The disk invariants with $\hat \phi_\bsi(\btau)$ inserted, versus with no insertion. The former corresponds to $\theta_\bsi$ while the later corresponds to $\Phi$ on the B-model.}
\label{fig:disk}
\end{center}
\end{figure}

We summarize the comparison of graph components in the following table.
\begin{center}
  \footnotesize
  \begin{tabular}{ |c | c | c |}
    \hline
    A-model  GW  & B-model spectral curve & Remark  \\ \hline
    \# of toric fixed pts       & \# of ramification points  &  dimension of Frobenius algebra \\
    \hline
    $\sqrt{\frac{1}{\Delta^\balpha(\btau)}}$ & $\frac{h^\balpha_1}{\sqrt{-2}}$ & length of canonical basis, \\ & &matching vertices\\
     \hline
    R-matrix $R_\balpha^{\ \bbeta}(z)$   & $\check R_{\balpha}^{\ \bbeta}(-z)$ & matching edges,\\ & &  dilaton leafs \\ \hline
    canonical coordinate $u^\balpha$& critical value $\hat x_{0,\balpha}$ & \\ \hline
   meromorphic form $\frac{\theta_\balpha^0}{\sqrt{-2}}$   & disk potential inserted by  $\hat\phi^\balpha(\tau)$& matching open leafs,\\&&see Figure \ref{fig:disk}\\
    \hline
    $\hat y d \hat x$ & disk  potential (no insertion)  & see Figure \ref{fig:disk} \\
   \hline
    \end{tabular}
\end{center}

\end{document}